\documentclass[leqno,CJK]{siamltex704}
\usepackage{amssymb,amsmath,graphicx,amscd,mathrsfs}
\usepackage{makecell}
\usepackage{color,xcolor,amsmath}
\usepackage{verbatim}
\usepackage{stmaryrd}
\usepackage{amsmath}
\usepackage{graphicx}
\usepackage{mathrsfs}
\usepackage{float}
\usepackage{amsfonts,amssymb}
\usepackage{dsfont}
\usepackage{pifont}
\usepackage{hyperref}
\usepackage{multirow}
\usepackage[section]{placeins}
\numberwithin{equation}{section}
\def\3bar{{|\hspace{-.02in}|\hspace{-.02in}|}}
\def\E{{\mathcal{E}}}
\def\T{{\mathcal{T}}}

\def\dQ{{\mathbb{Q}}}

\def\b0{\boldsymbol{0}}
\def\sumT{\sum_{T\in\mathcal{T}_h}}     

\def\bw{{\mathbf{w}}}

\def\bn{{\mathbf{n}}}

\def\bf{{\mathbf{f}}}

\newtheorem{example}{Example}[section]
\newtheorem{remark}{Remark}[section]
\newtheorem{algorithm1}{Weak Galerkin Algorithm}

 \newcommand{\eps}{\varepsilon}

 \newcommand{\Real}{\mathbb{R}}

 \newcommand{\trb}[1]{|\!|\!|#1|\!|\!|}
 \newcommand{\trbone}[1]{|\!|\!|#1|\!|\!|_1}

\allowdisplaybreaks 

\newcommand{\comm}[1]{{\color{red}#1}}

\begin{document}

\title{The Weak Galerkin Finite Element Method For The Steklov Eigenvalue Problem}

\author{
	Shusheng Li\thanks{School of Mathematics, Jilin University, Changchun, 130012, China (ssli22@mails.jlu.edu.cn).}
	\and
	Hehu Xie\thanks{ICMSEC, LSEC, NCMIS, Academy of Mathematics and Systems Science, Chinese Academy of Sciences, Beijing 100190, China, and School of Mathematical Sciences, University of Chinese Academy of Sciences, Beijing 100049, China (hhxie@lsec.cc.ac.cn).}
	\and 
	Qilong Zhai\thanks{School of Mathematics, Jilin University, Changchun, 130012, China (zhaiql@jlu.edu.cn).}
}

\maketitle
\begin{abstract}
This paper introduces a weak Galerkin (WG) finite element method to solve the Steklov eigenvalue problems, focusing on obtaining lower bounds of the eigenvalues. Compared with the existing work, our method can provide asymptotic lower bound approximations of eigenvalues with arbitrary high order convergence and no additional conditions are required. Moreover, through adjustments in stabilizer coefficients, we obtain guaranteed lower bound analysis for the weak Galerkin finite element method with arbitrary high degree. Numerical results demonstrate the accuracy and lower bound property of the numerical scheme.
\end{abstract}



\section{Introduction}
The Steklov eigenvalue problems with eigenvalue parameter appearing at the boundary have received extensive attention and have been applied to the fields of mathematics and physics, for example, stability of mechanical oscillators immersed in a viscous fluid \cite{MR1652238}, surface waves \cite{MR0054140}, the vibration modes of a structure in contact with an incompressible fluid \cite{MR1804656}, vibrations of a pendulum \cite{MR613954}, the antiplane shearing on a system of collinear faults under slip-dependent friction law \cite{MR2279256}, eigenoscillations of mechanical systems with boundary
conditions containing frequency \cite{MR1086767}, among many others.

Several research papers have explored various numerical techniques for approximating the eigenvalues and eigenfunctions associated with the Steklov eigenvalue problems. These studies can be summarized as follows. Bramble and Osborn \cite{1972Approximation} delved into the use of the conforming finite element methods to solve the Steklov eigenvalue problems. Andreev and Todorov \cite{MR2046179} focused their researches on finite element methods tailored to approximating the Steklov eigenvalue problems for second-order, self-adjoint, elliptic differential operators. Furthermore, Yang and Bi \cite{MR3202019} made contributions by introducing local $a$ $priori$ and $a$ $posteriori$ error estimates for the conforming finite element approximations of the Steklov eigenvalue problems. These estimates served as valuable tools for assessing the accuracy of numerical solutions. In \cite{MR4497827}, Wang et al. analyzed the use of the conforming virtual element methods for non-self-adjoint Steklov eigenvalue problems. They also provided an $a$ $priori$ error estimate for discrete eigenvalues and eigenfunctions. 

The reference to the minimum-maximum principle highlights a limitation of the conforming finite element methods, as they can only provide upper bounds of the eigenvalues. This limitation underscores the importance of obtaining lower bounds of the eigenvalues. Yang et al. \cite{MR2553141} studied the nonconforming finite element methods for the Steklov eigenvalue problems on convex domains. They specifically looked into methods for obtaining lower bounds of eigenvalues using different finite element types, such as the Crouzeix–Raviart element, the $Q_1^{rot}$ element and the $EQ_1^{rot}$ element. Their findings highlighted that the $EQ_1^{rot}$ element can get asymptotic lower bounds of all exact eigenvalues, while the Crouzeix–Raviart element and the $Q_1^{rot}$ element can only get lower bounds of the large eigenvalues, e.g., the upper bounds of the smallest eigenvalue can be obtained by these two methods. Furthermore, Li et al. \cite{MR3034819} studied the nonconforming finite element methods for the Steklov eigenvalue problems on the convex and concave domains. They introduced a novel element, the ECR element, and conducted  rigorous analysis of asymptotic lower bound approximations for the Steklov eigenvalues problems on the general domains. To provide rigorous lower eigenvalue  bounds  for  the  Steklov eigenvalue problems, You et al. \cite{MR3961991} introduced an enhanced version of an eigenvalue estimate algorithm designed to handle positive semidefinite bilinear forms. This algorithm extended the applicability of existing methods tailored for positive definite bilinear forms. In addition, Zhang et al. \cite{MR4218599} introduced a corrective approach to the Crouzeix-Raviart finite element eigenvalue approximations. They managed to obtain asymptotic lower bounds of eigenvalues for the Steklov eigenvalue problems featuring variable coefficients. Crucially, their work eliminated the need for specific conditions, such as singular eigenfunctions and large eigenvalues, thus enhancing the versatility. Unfortunately, the above work only considered lower bounds of eigenvalues with lower order convergence.

The main goal of this paper is to use the weak Galerkin (WG) finite element method to solve the Steklov eigenvalue problems and obtain lower bounds of eigenvalues. The WG method, originally proposed in \cite{MR2994424}, has been successfully applied to second order elliptic problems \cite{MR3325251, MR3767812, MR3223326}, Stokes equations \cite{MR4399131, MR3452926, MR3475512, MR4042721, MR3426142}, Helmholtz equations \cite{MR3407260, MR3784709}, biharmonic equations \cite{MR3366088}, Maxwell equations \cite{MR3394450}, Navier-Stokes equations \cite{MR3869651} and Stokes-Darcy problems \cite{MR4316143}. The WG method uses the special weak differential operators instead of the traditional differential operators and uses discontinuous piecewise polynomials as basis functions on polygonal or polyhedral finite element partitions. Therefore, the WG method can flexibly solve partial differential equations. 

In \cite{MR3919912}, Zhai et al. proposed the WG method for elliptic eigenvalue problems. A general framework was proposed and applied to elliptic eigenvalue problems. As long as (A1)-(A7) in \cite{MR3919912} are proved, the eigenvalues obtained by the WG method are asymptotic lower bounds of the exact eigenvalues. Subsequently, Carstensen et al. \cite{2020A} proposed a skeletal finite element method that can compute lower eigenvalue bounds. They rigorously demonstrated the existence of guaranteed lower eigenvalue bounds (GLB) for the Laplacian eigenvalue problems and conducted an analysis of the lowest-order case. The main contributions of this paper can be summarized as follows. Firstly, we use the WG method to solve the Steklov eigenvalue problems for the first time and extend the general framework in \cite{MR3919912} and the GLB in \cite{2020A} to the Steklov eigenvalue problems. Secondly, our method can obtain asymptotic lower bounds of the exact eigenvalues with arbitrary high order convergence, by adjusting the coefficients of stabilizer, we obtain the guaranteed lower bound analysis for the WG method with arbitrary high degree.

This paper is organized as follows. In Section 2, we introduce the WG scheme for the Steklov eigenvalue problems. In Section 3, we get error estimates of eigenvalues and eigenfunctions, then the asymptotic lower bounds with arbitrary high order convergence are obtained. The error estimates in $H^1$ and $L^2$ norms for the source problems corresponding to the Steklov eigenvalue problems are given in Section 4. In Section 5, we prove the GLB theorem which gives the conditions for guaranteed lower eigenvalue bounds. Numerical experiments are presented to demonstrate the accuracy and lower bound property of the WG scheme in Section 6.

\section{The weak Galerkin finite element method}
In this section, we propose the WG scheme for the Steklov eigenvalue problems and give the general notations.
\subsection{The Steklov eigenvalue problem}
In this paper, for simplicity, we consider the following equation:
\begin{equation}
\left\{
\begin{aligned}
-\Delta{u}+u &= 0,& &\quad \text{in }\Omega,&\\
\frac{\partial u}{\partial \mathbf{n}} &= \lambda{u},& &\quad \text{on }\partial\Omega,&
\end{aligned}
\right.
\label{2.1}
\end{equation}
where $\Omega$ is a bounded polygonal  or polyhedral domain in $\mathbb{R}^d$ ($d = 2, 3$), $\frac{\partial}{\partial \mathbf{n}}$ is the outward normal derivative of $\partial\Omega$.

The corresponding variational form of the equation (\ref{2.1}) is: Find $\lambda\in \mathbb{R}$ and $u \in H^1(\Omega)$ such that $\Vert u \Vert_{b}=1$ and
\begin{equation}
a(u,v)=\lambda b(u,v), ~~~\forall v\in H^{1}(\Omega),
\label{2.2}
\end{equation}
where 
\begin{align*}
&a(u,v)=\int_{\Omega}(\nabla u \nabla v+uv)dx,  \\
&b(u,v)=\int_{\partial \Omega}uvd\sigma,~~\Vert u \Vert_{b}=b(u,u)^{1/2},
\end{align*}
where $\sigma$ is the measure on $\partial \Omega$.

From \cite{MR0054140,1972Approximation}, the Steklov eigenvalue problem (\ref{2.1}) has an eigenvalue sequence $\{\lambda_{i}\}$:
$$0< \lambda_{1}\leq \lambda_{2} \leq \cdots \leq\lambda_{i}\leq \cdots,~~\lim\limits_{i \to \infty}\lambda_{i}=\infty,$$
with the corresponding eigenfunction sequence 
$$u_{1},u_{2},\cdots,u_{i},\cdots,$$ where $b(u_{i},u_{j})=\delta_{ij}.$

In this paper, we adopt the standard Sobolev space notation. Consider any open bounded domain $D$ with a Lipschitz continuous boundary in $\mathbb{R}^d$ ($d = 2, 3$). We denote the seminorm on this domain as $|\cdot|_{D}$. The norm and inner product are represented as $\|\cdot\|_D$ and
$(\cdot,\cdot)_D$, respectively. When $D=\Omega$, we will omit the subscript $D$. 

Consider $\T_h$ as a partition of the domain $\Omega$, where the cells in $\T_h$ are either polygons in two dimensions or polyhedra in three dimensions satisfying the regular assumptions specified in \cite{MR3223326}. We denote the set of all edges or faces of $T$ as $\mathcal{E}(T)$. The set of all edges or faces in  $\T_h$ is represented by $\mathcal{E}_h$, and $\E_h^0$ stands for the set of all interior edges or faces.
For each cell $T$ in $\T_h$, $h_T$ represents its diameter, and $h$ is the maximum diameter among all cells in $\T_h$. Let $C$ be a positive constant which is independent of the mesh size. We use the notation $a\lesssim b$ to denote $a\leq Cb$.

\subsection{The weak Galerkin scheme}
First we introduce the WG space for the Steklov eigenvalue problem (\ref{2.1}).
For any integer $k\ge 1$, we define
$$V_{h}=\{v=\{v_{0},v_{b}\}:v_{0}|_{T}\in P_{k}(T),v_{b}|_{e}\in P_{k}(e),\forall T\in \mathcal{T}_{h},\forall e\in\mathcal{E}_{h}\},$$
where $P_{k}(T)$ denotes the space of ploynomials whose degree is not more than $k$ on $T$, $P_{k}(e)$ denotes the space of polynomials whose degree is not more than $k$ on $e$.

For each cell $T\in \mathcal{T}_{h}$, we introduce the $L^{2}$ projection operator $Q_{0}$, which projects from $L^{2}(T)$ onto $P_{k}(T)$. Similarly, we denote $Q_{b}$ as the $L^{2}$ projection operator, but this one projects from $L^{2}(e)$ onto $P_{k}(e)$ for each edge/face $e \in \mathcal{E}_{h}$. Additionally, for each cell $T\in \mathcal{T}_{h}$, we define $\mathbb{Q}_{h}$ as the $L^{2}$ projection operator from $[L^{2}(T)]^d$ onto $[P_{k-1}(T)]^d$. We define $Q_{h} = \{Q_{0},Q_{b}\}$ to denote the projection operator from $H^1(\Omega)$ onto $V_{h}$.

\begin{definition}
For each $v_h\in V_h$, its weak gradient is defined as the following polynomial  $\nabla_w v_h|_T \in [P_{k-1}(T)]^d$ satisfying
\begin{eqnarray}\label{2.3}
(\nabla_w v_h,\bw)_T=-(v_0,\nabla\cdot\bw)_T+\langle v_b,\bw\cdot\bn
\rangle_{\partial T},\quad\forall\bw\in [P_{k-1}(T)]^d,
\end{eqnarray}
where $\bn$ denotes the outward unit normal vector.
\end{definition}

For any $v_{h},w_{h}\in V_h$, we define three bilinear forms on $V_h$
\begin{align*}
&s(v_{h},w_{h})=\gamma(h)\sumT h_T^{-1}\langle v_0-v_b,
w_0-w_b\rangle_{\partial T},\\
&a_w(v_{h},w_{h})=(\nabla_w v_{h},\nabla_w w_{h})+(v_{0},w_{0})+s(v_{h},w_{h}),\\
&b_w(v_{h},w_{h})=\langle v_{b},w_{b}\rangle_{\partial \Omega},
\end{align*}
where $\gamma(h)$ satisfies $\gamma(h) \leq 1$ and $\gamma(h) \rightarrow 0$ as $h \rightarrow 0$. The selection of $\gamma(h)$ is not unique, such as $\gamma(h)=h^{\varepsilon}$ or $-\frac{1}{\log(h)}$, where $\eps$ is a positive constant satisfying $0< \eps<1/4$. 

Next we give the WG algorithm of (\ref{2.1}).

\begin{algorithm1}
Find $(\lambda_h, u_h)\in\Real\times V_h$ such that $\|u_h\|_b=1$ and
\begin{eqnarray}\label{2.4}
a_w(u_h,v_h)=\lambda_h b_w(u_h,v_h),\quad\forall v_h\in V_h.
\end{eqnarray}
\end{algorithm1}

\section{Asymptotic lower bound}
In this section, we give some technical tools that are used later and the error estimates for the Steklov eigenvalue problem (\ref{2.2}). Finally, we obtain asymptotic lower bound approximations of the exact eigenvalues with arbitrary high order convergence.

\subsection{Technical tools}
Next we introduce several inequalities used in this paper, and the proofs can be found in \cite{MR3223326}.
\begin{lemma}
(Trace inequality). For each cell $T \in \mathcal{T}_{h}$ and each edge $e\subset \partial T$, we have
$$\Vert g \Vert_{e}^{2}\leq C(h_{T}^{-1}\Vert g \Vert_{T}^{2}+h_{T}\Vert \nabla g \Vert_{T}^{2}),\quad \forall g \in H^{1}(T).$$
\end{lemma}
\begin{lemma}
(Inverse inequality). Let $\varphi$ be a piecewise polynomial on  cell $T \in \mathcal{T}_{h}$, then we have
$$\Vert \nabla \varphi \Vert_{T}\leq Ch_{T}^{-1}\Vert \varphi \Vert_{T}.$$
\end{lemma}
\begin{lemma}
(Projection inequality). For any $\phi \in H^{k+1}(\Omega)$, we have
\begin{align*}
&\sum\limits_{T \in \mathcal{T}_{h}}\Vert \phi-Q_{0}\phi \Vert_{T}^{2}+\sum\limits_{T \in \mathcal{T}_{h}}h_{T}^{2}\Vert \nabla(\phi-Q_{0}\phi) \Vert_{T}^{2} \leq Ch^{2(k+1)}\Vert \phi \Vert_{k+1}^{2},\\
&\sum\limits_{T \in \mathcal{T}_{h}}\Vert \nabla \phi-\mathbb{Q}_{h}(\nabla \phi) \Vert_{T}^{2}+\sum\limits_{T \in \mathcal{T}_{h}}h_{T}^{2}\Vert \nabla(\nabla \phi-\mathbb{Q}_{h}(\nabla \phi)) \Vert_{T}^{2} \leq Ch^{2k}\Vert \phi \Vert_{k+1}^{2},\\
&\sum\limits_{T \in \mathcal{T}_{h}}\Vert \phi-Q_{b}\phi\Vert^{2}_{\partial T}\leq Ch^{2k+1}\Vert \phi\Vert_{k+1}^{2}.
\end{align*}
\end{lemma}

\begin{proof}
The proofs of the first two inequalities can be found in \cite{MR3223326}, we only give the proof of the third one. From the trace inequality and the first projection inequality, we have
\begin{align*}
    \sum\limits_{T \in \mathcal{T}_{h}}\Vert \phi-Q_{b}\phi\Vert^{2}_{\partial T}&\leq \sum\limits_{T \in \mathcal{T}_{h}}\Vert \phi-Q_{0}\phi\Vert^{2}_{\partial T}  \\
    &\leq \sum\limits_{T \in \mathcal{T}_{h}}C(h_{T}^{-1}\Vert \phi-Q_{0}\phi\Vert^{2}_{T}+h_{T}\Vert \nabla(\phi-Q_{0}\phi)\Vert^{2}_{T}) \\
    &\leq Ch^{2k+1}\Vert \phi\Vert_{k+1}^{2},
\end{align*}
which completes the proof.
\end{proof}

The following commutativity property plays a crucial role in the theoretical analysis, and the proof can be found in \cite[Lemma 5.1]{MR3325251}.

\begin{lemma}
\label{lemma2.2}
For any cell \:$T\in \mathcal{T}_{h}$\:, we have
$$\nabla_{w}(Q_{h}\phi)=\mathbb{Q}_{h}(\nabla\phi), \quad \forall \phi \in H^{1}(T).$$
\end{lemma}

\subsection{Error estimate}

Let $V_c=H^1(\Omega)$, and $V=V_c+V_h$.
For any $u \in V$, we define the following norm
\begin{eqnarray*}
\|u\|_V^2= \sumT \Big(\|\nabla u_0\|_T^2+\|u_0\|_T^2+ h_T^{-1}\|u_0-u_b\|^2_{\partial T}\Big).
\end{eqnarray*}
Here $u_{0}$ represents the value of $u$ on the interior of each cell $T$ in the partition $\mathcal{T}_{h}$, while $u_{b}$ represents the value of $u$ on each edge or face $e \in \mathcal{E}_{h}$.

Next we verify the (A1)-(A7) in Section 2 of \cite{MR3919912} to prove lower bound property of the eigenvalues.

First we verify (A1):
$a(\cdot,\cdot)$ and $a_w(\cdot,\cdot)$ are symmetric, and
\begin{align*}
    &a(u,u)\geq \gamma_c\Vert u\Vert_V^2,\quad \forall u\in V_c,\\
    &a_w(u_h,u_h)\geq \gamma(h)\Vert u_h\Vert_V,\quad \forall u_h \in V_h,
\end{align*}
where $\gamma_c$ and $\gamma(h)$ are two positive constants for given $h$.

It is easy to verify that $a(\cdot,\cdot)$ is symmetric, coercive and continuous, and $a_{w}(\cdot,\cdot)$ is symmetric and continuous. Therefore, we only need to prove that $a_{w}(\cdot,\cdot)$ is coercive.

\begin{lemma}
\label{lemma2.3}
For any \:$u_{h} \in V_{h}$\:, we have
$$a_{w}(u_{h},u_{h})\gtrsim \gamma(h)\Vert u_{h} \Vert_{V}^{2}.$$
\end{lemma}
\begin{proof}
Similar to \cite[Lemma 3.2]{MR3919912}, we can obtain
\begin{align*}
&\sum\limits_{T\in \mathcal{T}_{h}}\Vert \nabla u_{0} \Vert_{T}^{2}\leq C\gamma(h)^{-1}a_{w}(u_{h},u_{h}),\\
&\sum\limits_{T\in \mathcal{T}_{h}}h_{T}^{-1}\Vert u_{0}-u_{b} \Vert_{\partial T}^{2}\leq C\gamma(h)^{-1}a_{w}(u_{h},u_{h}),\\
&\sum\limits_{T\in \mathcal{T}_{h}}\Vert u_{0} \Vert_{T}^{2}\leq \gamma(h)^{-1}a_{w}(u_{h},u_{h}).
\end{align*}
Combining the above three inequalities, we have
$$a_{w}(u_{h},u_{h})\gtrsim \gamma(h)\Vert u_{h} \Vert_{V}^{2},$$
which completes the proof.
\end{proof}

We define two operators $K, K_h$ as follows
\begin{align*}
&K:~V_{c}\rightarrow V_{c},~~a(Kf,v)=b(f,v), ~~~\forall v\in V_{c},\\
&K_{h}:~V_{h}\rightarrow V_{h},~~a_{w}(K_{h}f_{h},v_{h})=b_{w}(f_{h},v_{h}), ~~~\forall v_{h}\in V_{h}.
\end{align*}
It is easy to prove that $K$ and $K_{h}$ are well-defined by the Lax-Milgram theorem.

Next we verify (A2): $K, K_h$ are compact.

\begin{lemma}
\label{lemma2.4}
$K,K_{h}$ are compact operators.
\end{lemma}
\begin{proof}
From \cite[Proposition 4.4]{MR1804656} and \cite[(4.10)]{1972Approximation}, we have $\Vert Kf \Vert_{1+r}\leq C\Vert f \Vert_{\frac{1}{2}, \partial \Omega}$, where $1/2<r \leq 1$. It follows from the trace inequality that $\Vert f\Vert_{\frac{1}{2}, \partial \Omega}\leq C\Vert f\Vert_{1}$. And $H^{1+r}(\Omega)$ is compact embedded into $H^{1}(\Omega)$, so $K$ is a compact operator. For the operator $K_{h}$, since it is a finite rank operator, it is only necessary to prove that it is a bounded linear operator. From Lemma \ref{lemma2.3}, we have
$$\gamma(h)\Vert K_{h}f_{h} \Vert_{V}^{2}\lesssim a_{w}(K_{h}f_{h},K_{h}f_{h})=b_{w}(f_{h},K_{h}f_{h}).$$
It follows from \cite[Lemma 3.11]{MR4567823} that
\begin{align*}
&b_{w}(K_{h}f_{h},f_{h})\leq \Vert K_{h}f_{h} \Vert_{b} \Vert f_{h} \Vert_{b} \leq C  \trbone{K_{h}f_{h}} \trbone{f_{h}},
\end{align*}
where 
$$\trbone{v}=\sum\limits_{T\in \mathcal{T}_{h}}\Vert \nabla_{w}v \Vert^2_{T}+\sum\limits_{T\in \mathcal{T}_{h}}\Vert v_{0}\Vert^2_{T}+\sum\limits_{T\in \mathcal{T}_{h}}h_{T}^{-1}\Vert v_{0}-v_{b} \Vert^2_{\partial T}, ~\forall v \in V_{h}.$$
According to \cite[Lemma 5.3]{MR3325251}, there exist two positive constants $C_1$ and $C_2$ such that the following inequality holds
\begin{align*}
&C_{1}\Vert v \Vert_{V}\leq \trbone{v} \leq C_{2}\Vert v \Vert_{V}, ~~~\forall v \in V_{h}.
\end{align*}
Thus, we have $\Vert K_{h}f_{h} \Vert_{V}\leq C\Vert f_{h} \Vert_{V}$, so \:$K_{h}$\; is a bounded linear operator. Then we complete the proof.
\end{proof}

Next we verify (A3): there exists a bounded linear operator $Q_h: V\rightarrow V_h$ satisfying
\begin{align*}
    &Q_h v_h=v_h,\quad \forall v_h\in V_h,\\
    &b(Q_h w,v_h)=b(w,v_h),\quad \forall w\in V,v_h\in V_h.
\end{align*}

Obviously, $Q_{h}v_{h}=v_{h}, \forall v_{h}\in V_{h}$, and
$$b(w,v_{h})=\langle w_{b},v_{b} \rangle_{\partial\Omega}=\langle Q_{b}w_{b},v_{b} \rangle_{\partial\Omega}=b(Q_{h}w,v_{h}).$$
Then (A3) is verified.

Let $X$ be the $L^{2}(\partial\Omega)$ space. Denote
\begin{align*}
&\delta_{h,\mu}=\sup\limits_{v\in R(E_{\mu}(K)),\atop \Vert v \Vert_{V}=1}\Vert v-Q_{h}v \Vert_{V},\\
&\delta_{h,\mu}'=\sup\limits_{v\in R(E_{\mu}(K)),\atop \Vert v \Vert_{V}=1}\Vert v-Q_{h}v \Vert_{X}.
\end{align*}
The definition of $R(E_{\mu}(K))$ can be found in Section 2 of \cite{MR3919912}. It is the eigenspace corresponding to the eigenvalue $\mu$.

Denote
\begin{align*}
&e_{h,\mu}=\Vert (K-K_{h}Q_{h})|_{R(E_{\mu}(K)} \Vert_{V},  \\
&e_{h,\mu}'=\Vert (\Pi_{0}K-\Pi_{0}K_{h}Q_{h})|_{R(E_{\mu}(K))} \Vert_{X},
\end{align*}
where $\Pi_{0}$ is the projection operator from $V$ onto $X$ under the inner product $b(\cdot,\cdot)$.

Next we verify (A4) and (A5):
\begin{align*}
    &(A4): e_{h,\mu}\rightarrow 0 ~{\rm as} ~h \rightarrow 0, {\rm and} ~\delta_{h,\mu}\gamma(h)^{-1}\rightarrow 0~{\rm as} ~h \rightarrow 0,\\
    &(A5): e_{h,\mu}'\rightarrow 0 ~{\rm as} ~h \rightarrow 0, {\rm and} ~\delta_{h,\mu}'\gamma(h)^{-1}\rightarrow 0~{\rm as} ~h \rightarrow 0.
\end{align*}

\begin{lemma}
\label{lemma2.5}
Assume $R(E_{\mu}(K))\subset H^{k+1}(\Omega)$, we can establish the following estimates
\begin{align*}
&\delta_{h,\mu}\lesssim h^{k},\\
&\delta_{h,\mu}'\lesssim h^{k+\frac{1}{2}}.
\end{align*}
\end{lemma}
\begin{proof}
Assume $u\in R(E_{\mu}(K))\subset H^{k+1}(\Omega)$. From the trace inequality and the projection inequality, we can obtain
\begin{align*}
&\Vert u-Q_{h}u \Vert_{V}^{2} \\
=&\sum\limits_{T\in \mathcal{T}_{h}}\Vert \nabla(u-Q_{0}u) \Vert_{T}^{2}+\sum\limits_{T\in \mathcal{T}_{h}}\Vert (u-Q_{0}u) \Vert_{T}^{2}+\sum\limits_{T\in \mathcal{T}_{h}}h_{T}^{-1}\Vert Q_{0}u-Q_{b}u \Vert_{\partial T}^{2}  \\
\leq&\sum\limits_{T\in \mathcal{T}_{h}}\Vert \nabla(u-Q_{0}u) \Vert_{T}^{2}+\sum\limits_{T\in \mathcal{T}_{h}}\Vert (u-Q_{0}u) \Vert_{T}^{2}+\sum\limits_{T\in \mathcal{T}_{h}}h_{T}^{-1}\Vert Q_{0}u-u \Vert_{\partial T}^{2}  \\
\leq&\sum\limits_{T\in \mathcal{T}_{h}}\Vert \nabla(u-Q_{0}u) \Vert_{T}^{2}+\sum\limits_{T\in \mathcal{T}_{h}}\Vert (u-Q_{0}u) \Vert_{T}^{2}\\
&+\sum\limits_{T\in \mathcal{T}_{h}}C(h_{T}^{-2}\Vert Q_{0}u-u \Vert_{T}^{2}+\Vert \nabla(Q_{0}u-u) \Vert_{T}^{2})  \\
\lesssim&h^{2k},
\end{align*}
and
$$\Vert u-Q_{h}u \Vert_{X}^{2}\leq\sum\limits_{T\in \mathcal{T}_{h}}\Vert (u-Q_{b}u) \Vert_{\partial T}^{2}\lesssim h^{2k+1}.$$
Then we complete the proof.
\end{proof}

Consider the source problem corresponding to the Steklov eigenvalue problem (\ref{2.1}):
\begin{equation}
\left\{
\begin{aligned}
-\Delta{u}+u &= 0,& &\quad\text{in }\Omega,&\\
\frac{\partial u}{\partial \mathbf{n}} &= f,& &\quad \text{on }\partial\Omega,&\\
\end{aligned}
\right.
\label{2.5}
\end{equation}
where $f\in L^{2}(\partial \Omega)$. The corresponding WG scheme is: Find $u_{h}\in V_{h}$ such that
\begin{equation}
a_{w}(u_{h},v_{h})=b(f,v_{h}), \quad\forall v_{h}\in V_{h}.
\label{2.6}
\end{equation}
And we have the following two equations:
\begin{equation}
\left\{
\begin{aligned}
a(Kf,v)=b(f,v), \quad\forall v\in V_{c},\\
a(u,v)=b(f,v), \quad\forall v\in V_{c},
\end{aligned}
\right.
\label{2.7}
\end{equation}
\begin{equation}
\left\{
\begin{aligned}
a_{w}(K_{h}Q_{h}f,v_{h})=b(f,v_{h}), \quad\forall v_{h}\in V_{h},\\
a_{w}(u_{h},v_{h})=b(f,v_{h}), \quad\forall v_{h}\in V_{h}.
\end{aligned}
\right.
\label{2.8}
\end{equation}
From (\ref{2.7}) and (\ref{2.8}), we can get $Kf=u$ and
\:$K_{h}Q_{h}f=u_{h}$\:, respectively. The $H^1$ and $L^2$ error estimates of (\ref{2.6}) are given in the next section, and the details can be seen in Theorems \ref{theorem3.3} and \ref{theorem3.4}. Hence we have the following lemma.

\begin{lemma}
\label{lemma2.6}
Assume $R(E_{\mu}(K))\subset H^{k+1}(\Omega)$, the following error estimates hold
\begin{align*}
&e_{h,\mu}\lesssim \gamma(h)^{-1}h^{k}, \\
&e_{h,\mu}'\lesssim \gamma(h)^{-1}h^{k+\frac{r}{2}},
\end{align*}
where $1/2<r\leq 1$.
\end{lemma}

According to \cite[Theorem 2.1, Theorem 2.2]{MR3919912}, we have the following theorem. The details can be seen in \cite[Theorem 7.1]{Ivo1991Eigenvalue}.

\begin{theorem}
\label{theorem2.7}
Let $\lambda$ be an eigenvalue of $(\ref{2.2})$ with multiplicity $m$, $\{u_{j}\}_{j=1}^{m}$ represent a basis of the corresponding m-dimensional eigenspace $R(E_{\mu}(K))\subset H^{k+1}(\Omega)$. Similarly, let $\{\lambda_{j,h}\}_{j=1}^{m}$ denote the eigenvalues of $(\ref{2.4})$, $\{u_{j,h}\}_{j=1}^{m}$ be a basis of the corresponding eigenspace $R(E_{\mu,h}(K_{h}))$. For each $j=1,\cdots,m$, it can be established that there exists an eigenfunction $u_{j}\in R(E_{\mu}(K))$ such that
\begin{align*}
&\Vert u_{j}-u_{j,h} \Vert_{V}\lesssim \gamma(h)^{-1}h^{k},\\
&\Vert u_{j}-u_{j,h} \Vert_{X}\lesssim \gamma(h)^{-1}h^{k+\frac{r}{2}}.
\end{align*}
\end{theorem}

Given that the parameter $\gamma(h)$ is chosen as either $\gamma(h)=h^{\varepsilon}$, where $0< \eps<1/4$, or $\gamma(h)=-\frac{1}{\log(h)}$, the conclusions presented in Theorem $\ref{theorem2.7}$ are as follows
\begin{align*}
&\Vert u_{j}-u_{j,h} \Vert_{V}\lesssim h^{k-\varepsilon},\\
&\Vert u_{j}-u_{j,h} \Vert_{X}\lesssim h^{k+\frac{r}{2}-\varepsilon},
\end{align*}
or
\begin{align*}
&\Vert u_{j}-u_{j,h} \Vert_{V}\lesssim -\log(h)h^{k},\\
&\Vert u_{j}-u_{j,h} \Vert_{X}\lesssim -\log(h)h^{k+\frac{r}{2}}.
\end{align*}

Denote $\varepsilon_{h,u}=a(u,u)-a_{w}(Q_{h}u,Q_{h}u).$ Next we verify (A6): $\forall \mu\in\sigma_K$, $u\in R(E_{\mu}(K))$, $\varepsilon_{h,\mu}\rightarrow0$ as $h\rightarrow0$, where $\sigma_K$ is the spectrum of $K$. 

\begin{lemma}
\label{lemma2.8}
Let $u\in H^{k+1}(\Omega)$, the following estimate holds
$$\vert \varepsilon_{h,u} \vert \lesssim h^{2k}.$$
\end{lemma}
\begin{proof}
It follows from Lemma $\ref{lemma2.2}$ that
\begin{flalign*}
&\vert \varepsilon_{h,u} \vert  \\
=&\left\vert \sum\limits_{T\in \mathcal{T}_{h}}\Vert \nabla u \Vert_{T}^{2}+\sum\limits_{T\in \mathcal{T}_{h}}\Vert u \Vert_{T}^{2}-\sum\limits_{T\in \mathcal{T}_{h}}\Vert \nabla_{w}Q_{h}u \Vert_{T}^{2}-\sum\limits_{T\in \mathcal{T}_{h}}\Vert Q_{0}u \Vert_{T}^{2}-s(Q_{h}u,Q_{h}u) \right\vert            &\nonumber\\
=&\left\vert \sum\limits_{T\in \mathcal{T}_{h}}\Vert \nabla u \Vert_{T}^{2}+\sum\limits_{T\in \mathcal{T}_{h}}\Vert u \Vert_{T}^{2}-\sum\limits_{T\in \mathcal{T}_{h}}\Vert \mathbb{Q}_{h}\nabla u \Vert_{T}^{2}-\sum\limits_{T\in \mathcal{T}_{h}}\Vert Q_{0}u \Vert_{T}^{2}-s(Q_{h}u,Q_{h}u) \right\vert            &\nonumber\\
\leq&\sum\limits_{T\in \mathcal{T}_{h}}\Vert \nabla u-\mathbb{Q}_{h}\nabla u \Vert_{T}^{2}+\sum\limits_{T\in \mathcal{T}_{h}}\Vert u-Q_{0}u \Vert_{T}^{2}+\sum\limits_{T\in \mathcal{T}_{h}}\gamma(h)h_{T}^{-1}\Vert Q_{0}u-u \Vert_{\partial T}^{2}          &\nonumber\\
\leq&\sum\limits_{T\in \mathcal{T}_{h}}\Vert \nabla u-\mathbb{Q}_{h}\nabla u \Vert_{T}^{2}+\sum\limits_{T\in \mathcal{T}_{h}}\Vert u-Q_{0}u \Vert_{T}^{2}\\
&+\sum\limits_{T\in \mathcal{T}_{h}}C\gamma(h)(h_{T}^{-2}\Vert Q_{0}u-u \Vert_{T}^{2}+\Vert \nabla(Q_{0}u-u) \Vert_{T}^{2})          &\nonumber\\
\lesssim& h^{2k}.  &\nonumber
\end{flalign*}
This completes the proof.
\end{proof}

\begin{theorem}
\label{theorem2.9}
Let $\lambda$ be an eigenvalue of $(\ref{2.2})$ with multiplicity $m$, $\{u_{j}\}_{j=1}^{m}$ represent a basis of the corresponding m-dimensional eigenspace $R(E_{\mu}(K))\subset H^{k+1}(\Omega)$. Let $\{\lambda_{j,h}\}_{j=1}^{m}$ denote the eigenvalues of $(\ref{2.4})$, $\{u_{j,h}\}_{j=1}^{m}$ be a basis of the corresponding eigenspace $R(E_{\mu,h}(K_{h}))$. When the parameter $h$ becomes sufficiently small, for each $j=1,\cdots,m$, the following error estimate holds
$$\vert \lambda-\lambda_{j,h} \vert \lesssim \gamma(h)^{-2}h^{2k}.$$
\end{theorem}

\begin{proof}
It follows from \cite[Theorem 2.3]{MR3919912} that
$$\vert \lambda-\lambda_{j,h} \vert \leq C(\varepsilon_{h,u_{j}}+e_{h,\lambda^{-1}}^{2}+e_{h,\lambda^{-1}}'^{2}+\delta_{h,\lambda^{-1}}^{2}\gamma(h)^{-2}+
\delta_{h,\lambda^{-1}}'^{2}\gamma(h)^{-2}).$$
From Lemmas $\ref{lemma2.5}$, $\ref{lemma2.6}$ and $\ref{lemma2.8}$, we have
\begin{align*}
&\delta_{h,\lambda^{-1}}\lesssim h^{k}, ~\delta_{h,\lambda^{-1}}'\lesssim h^{k+\frac{1}{2}},\\
&e_{h,\lambda^{-1}}\lesssim \gamma(h)^{-1}h^{k}, ~e_{h,\lambda^{-1}}'\lesssim \gamma(h)^{-1}h^{k+\frac{r}{2}},\\
&\vert \varepsilon_{h,u} \vert \lesssim h^{2k},
\end{align*}
which implies that $$\vert \lambda-\lambda_{j,h} \vert \lesssim \gamma(h)^{-2}h^{2k}.$$
Then we complete the proof.
\end{proof}

Given that the parameter $\gamma(h)$ is chosen as either $\gamma(h)=h^{\varepsilon}$, where $0< \eps<1/4$, or $\gamma(h)=-\frac{1}{\log(h)}$, the conclusion presented in Theorem $\ref{theorem2.9}$ is as follows
$$\vert \lambda-\lambda_{j,h} \vert \lesssim h^{2k-2\varepsilon},$$
or
$$\vert \lambda-\lambda_{j,h} \vert \lesssim {\log(h)}^2h^{2k}.$$

Next we verify (A7) to prove the lower bound property. (A7): Suppose $(\lambda,u)$ is an eigenpair of $(\ref{2.2})$, $(\lambda_{h},u_{h})$ is an eigenpair of $(\ref{2.4})$, there holds $\varepsilon_{h,u}\geq \lambda_{h}\Vert u-u_{h} \Vert_{X}^{2}.$
The following estimates are crucial, and the proof can be found in \cite[Theorem 2.1]{MR3120579}.

\begin{lemma}
\label{lemma2.10}
The exact eigenfunction $u$ of the Steklov eigenvalue problem $(\ref{2.2})$ exhibits lower bounds for its convergence rates as follows
\begin{align*}
&\sum\limits_{T\in \mathcal{T}_{h}}\Vert \nabla u-\mathbb{Q}_{h}\nabla u \Vert_{T}^{2}\gtrsim h^{2k},\\
&\sum\limits_{T\in \mathcal{T}_{h}}\Vert u-Q_{0}u \Vert_{T}^{2}\gtrsim h^{2(k+1)}.
\end{align*}
\end{lemma}

\begin{lemma}
Let $(\lambda,u)$ be an eigenpair of $(\ref{2.2})$, $(\lambda_{h},u_{h})$ be an eigenpair of $(\ref{2.4})$. For sufficiently small $h$, suppose $\gamma(h)\ll 1$, then we have
$$\varepsilon_{h,u}\geq \lambda_{h}\Vert u-u_{h} \Vert_{X}^{2}.$$
\end{lemma}
\begin{proof}
It follows from Lemma $\ref{lemma2.2}$ that
\begin{flalign*}
&\varepsilon_{h,u}\\
=&a(u,u)-a_{w}(Q_{h}u,Q_{h}u)          &\nonumber\\
=&\sum\limits_{T\in \mathcal{T}_{h}}\Vert \nabla u \Vert_{T}^{2}+\sum\limits_{T\in \mathcal{T}_{h}}\Vert u \Vert_{T}^{2}-\sum\limits_{T\in \mathcal{T}_{h}}\Vert \nabla_{w}Q_{h}u \Vert_{T}^{2}-\sum\limits_{T\in \mathcal{T}_{h}}\Vert Q_{0}u \Vert_{T}^{2}-s(Q_{h}u,Q_{h}u)           &\nonumber\\
=&\sum\limits_{T\in \mathcal{T}_{h}}\Vert \nabla u-\mathbb{Q}_{h}\nabla u \Vert_{T}^{2}+\sum\limits_{T\in \mathcal{T}_{h}}\Vert u-Q_{0}u \Vert_{T}^{2}-\sum\limits_{T\in \mathcal{T}_{h}}\gamma(h)h_{T}^{-1}\Vert Q_{0}u-Q_{b}u \Vert_{\partial T}^{2}.
\end{flalign*}
Combining Lemma $\ref{lemma2.10}$ with the trace inequality, we can obtain
\begin{align*}
&\sum\limits_{T\in \mathcal{T}_{h}}\Vert \nabla u-\mathbb{Q}_{h}\nabla u \Vert_{T}^{2}\gtrsim h^{2k},\\
&\sum\limits_{T\in \mathcal{T}_{h}}\Vert u-Q_{0}u \Vert_{T}^{2}\gtrsim h^{2k+2},\\
&\sum\limits_{T\in \mathcal{T}_{h}}\gamma(h)h_{T}^{-1}\Vert Q_{0}u-Q_{b}u \Vert_{\partial T}^{2}\lesssim \gamma(h)h^{2k}.
\end{align*}
When $h$ is sufficiently small, $\gamma(h)\ll 1$, then we have
$$\varepsilon_{h,u}\gtrsim h^{2k}.$$
It follows from Theorem $\ref{theorem2.7}$ that
$$\lambda_{h}\Vert u-u_{h} \Vert_{X}^{2}\lesssim \gamma(h)^{-2} h^{2k+r}.$$
Thus, for sufficiently small $h$, the following inequality is satisfied
$$\varepsilon_{h,u}\geq \lambda_{h}\Vert u-u_{h} \Vert_{X}^{2}.$$
This completes the proof.
\end{proof}

From \cite[Theorem 2.4]{MR3919912}, we have the following theorem.
\begin{theorem}
\label{theorem2.15}
Let $(\lambda,u)$ be an eigenpair of $(\ref{2.2})$, $(\lambda_{h},u_{h})$ be an eigenpair of $(\ref{2.4})$, assume $\gamma(h)\rightarrow 0$ as $h\rightarrow 0$, for sufficiently small $h$, we have
$$\lambda\geq \lambda_{h}.$$
\end{theorem}

Combining Theorems \ref{theorem2.9} and \ref{theorem2.15}, we can obtain asymptotic lower bound approximations of the exact eigenvalues.

\section{The source problem error estimate}
In this section, we derive the $H^1$ and $L^2$ error estimates of the WG scheme (\ref{2.6}). First, we recall the source problem and the corresponding numerical scheme:
\begin{equation}
\tag{3.1}
\left\{
\begin{aligned}
-\Delta{u}+u &= 0,& &\quad\text{in }\Omega,&\\
\frac{\partial u}{\partial \mathbf{n}} &= f,& &\quad \text{on }\partial\Omega,&\\
\end{aligned}
\right.
\end{equation}
where $f\in L^{2}(\partial \Omega)$. The corresponding WG scheme is: Find $u_{h}\in V_{h}$ such that
\begin{equation}
\tag{3.2}
a_{w}(u_{h},v_{h})=b(f,v_{h}), \quad\forall v_{h}\in V_{h}.
\end{equation}

\subsection{Error equation}
Let $u$ represent the solution of (\ref{2.5}), and $u_h$ denote the solution of (\ref{2.6}). Denote
$$e_{h}=Q_{h}u-u_{h}=\{e_{0},e_{b}\}=\{Q_{0}u-u_{0},Q_{b}u-u_{b}\},$$
then $e_{h}$ satisfies the following lemma.

\begin{lemma}
\label{lemma3.1}
For any $v_h \in V_{h}$, we have 
\begin{equation*}
a_{w}(e_{h},v_{h})=\ell(u,v_{h})+s(Q_{h}u,v_{h}),
\end{equation*}
where
\begin{equation*}
\ell(u,v_{h})=\sum\limits_{T\in \mathcal{T}_{h}}\langle (\nabla u-\mathbb{Q}_{h}\nabla u)\cdot \mathbf{n},v_{0}-v_{b}\rangle_{\partial T}.
\end{equation*}
\end{lemma}

\begin{proof}
It follows from the definition of the weak gradient ($\ref{2.3}$) and Lemma $\ref{lemma2.2}$ that
\begin{flalign}
\label{3.1}
&\sum\limits_{T\in \mathcal{T}_{h}}(\nabla_{w}Q_{h}u,\nabla_{w}v_{h})_{T} \nonumber\\
=&\sum\limits_{T\in \mathcal{T}_{h}}(\mathbb{Q}_{h}\nabla u,\nabla_{w}v_{h})_{T}  \nonumber\\
=&-\sum\limits_{T\in \mathcal{T}_{h}}(\nabla \cdot \mathbb{Q}_{h}\nabla u,v_{0})_{T}+\sum\limits_{T\in \mathcal{T}_{h}}(\mathbb{Q}_{h}(\nabla u)\cdot \mathbf{n},v_{b})_{T}   \nonumber\\
=&\sum\limits_{T\in \mathcal{T}_{h}}(\mathbb{Q}_{h}\nabla u,\nabla v_{0})_{T}-\sum\limits_{T\in \mathcal{T}_{h}}\langle\mathbb{Q}_{h}(\nabla u)\cdot \mathbf{n},v_{0}-v_{b}\rangle_{\partial T}   \nonumber\\
=&\sum\limits_{T\in \mathcal{T}_{h}}(\nabla u,\nabla v_{0})_{T}-\sum\limits_{T\in \mathcal{T}_{h}}\langle\mathbb{Q}_{h}(\nabla u)\cdot \mathbf{n},v_{0}-v_{b}\rangle_{\partial T}.
\end{flalign}
For the source problem (\ref{2.5}), we use $v_{0}$ and $v_{b}$ as test functions, respectively. Then we have
\begin{flalign}
\label{3.2}
&-\sum\limits_{T\in \mathcal{T}_{h}}(\Delta u,v_{0})_{T}+\sum\limits_{T\in \mathcal{T}_{h}}(u,v_{0})_{T}  \nonumber\\
=&\sum\limits_{T\in \mathcal{T}_{h}}(\nabla u,\nabla v_{0})_{T}+\sum\limits_{T\in \mathcal{T}_{h}}(u,v_{0})_{T}-\sum\limits_{T\in \mathcal{T}_{h}}\langle \nabla u \cdot \mathbf{n},v_{0}\rangle_{\partial T}=0,
\end{flalign}
and
\begin{flalign}
\label{3.3}
&\sum\limits_{T\in \mathcal{T}_{h}}\langle \nabla u \cdot \mathbf{n},v_{b}\rangle_{\partial T}=\langle \nabla u \cdot \mathbf{n},v_{b}\rangle_{\partial \Omega}=\langle f,v_{b}\rangle_{\partial \Omega}.
\end{flalign}
From $(\ref{3.2})$ and $(\ref{3.3})$, we arrive at
\begin{flalign}
\label{3.4}
\sum\limits_{T\in \mathcal{T}_{h}}(\nabla u,\nabla v_{0})_{T}+\sum\limits_{T\in \mathcal{T}_{h}}(u,v_{0})_{T}-\sum\limits_{T\in \mathcal{T}_{h}}\langle \nabla u \cdot \mathbf{n},v_{0}-v_{b}\rangle_{\partial T}=\langle f,v_{b}\rangle_{\partial \Omega}.
\end{flalign}
Combining $(\ref{3.1})$ with $(\ref{3.4})$, we can obtain
\begin{flalign*}
&a_{w}(Q_{h}u,v_{h})    \\
=&\sum\limits_{T\in \mathcal{T}_{h}}(Q_{0}u-u,v_{0})_{T}+\langle f,v_{b}\rangle_{\partial \Omega}+s(Q_{h}u,v_{h})+\ell(u,v_{h}).
\end{flalign*}
Notice that $u_{h}$ is the solution that satisfies the following equation
$$a_{w}(u_{h},v_{h})=\langle f,v_{h}\rangle_{\partial \Omega}=\langle f,v_{b}\rangle_{\partial \Omega},$$
and $u$ satisfies the following equation $$\sum\limits_{T\in \mathcal{T}_{h}}(Q_{0}u-u,v_{0})_{T}=0.$$
Thus, we have $$a_{w}(e_{h},v_{h})=\ell(u,v_{h})+s(Q_{h}u,v_{h}), \quad\forall v_{h} \in V_{h},$$
which completes the proof.
\end{proof}

\subsection{$H^1$ and $L^2$ error estimates}
Similarly, from \cite[Lemma A.2, Theorem A.1]{MR3919912}, we have the following estimates.
\begin{lemma}
\label{lemma3.2}
The following estimates hold for any $v_{h} \in V_{h}$ and $w \in H^{k+1}(\Omega)$
\begin{align*}
&\vert s(Q_{h}w,v_{h})\vert \leq C \gamma(h) h^{k}\Vert w \Vert_{k+1}\Vert v_{h} \Vert_{V},\\
&\vert \ell(w,v_{h})\vert \leq C h^{k}\Vert w \Vert_{k+1} \Vert v_{h} \Vert_{V}.
\end{align*}
\end{lemma}

\begin{theorem}
\label{theorem3.3}
Let $u\in H^{k+1}(\Omega)$ represent the exact solution of the source problem $(\ref{2.5})$, $u_{h}$ denote the numerical solution of the WG scheme $(\ref{2.6})$. Then we have the following estimate 
$$\Vert u-u_{h} \Vert_{V}\leq C\gamma(h)^{-1}h^{k}\Vert u \Vert_{k+1}.$$
\end{theorem}

To derive $L^2$ error estimate, we consider the dual problem of (\ref{2.5}):
\begin{equation}
\left\{
\begin{aligned}
-\Delta{\phi}+\phi&=0,& &\quad\text{in }\Omega,&\\
\frac{\partial \phi}{\partial \mathbf{n}}&=e_b,& &\quad\text{on }\partial\Omega.&
\end{aligned}
\right.
\label{3.5}
\end{equation}
Assume the dual problem (\ref{3.5}) has $H^{1+r/2}(\Omega)$-regularity property, and the priori estimate holds true which can be found in \cite[Proposition 4.4]{MR1804656} and \cite[(4.10)]{1972Approximation}:
$$\Vert\phi\Vert_{1+r/2}\leq C\Vert e_b\Vert_X,$$
where $1/2<r\leq 1$.

\begin{theorem}
\label{theorem3.4}
Let $u\in H^{k+1}(\Omega)$ be the exact solution of $(\ref{2.5})$, $u_{h}$ be the numerical solution of the WG scheme $(\ref{2.6})$. Then we have the following estimate
$$\Vert u-u_{h} \Vert_{X}\leq C\gamma(h)^{-1}h^{k+\frac{r}{2}}\Vert u \Vert_{k+1}.$$
\end{theorem}
\begin{proof}
Testing $(\ref{3.5})$ by $e_{0}$ and $e_{b}$, respectively, we obtain
\begin{flalign}
\label{3.6}
&\sum\limits_{T\in \mathcal{T}_{h}}-(\Delta \phi,e_{0})_{T}+\sum\limits_{T\in \mathcal{T}_{h}}(\phi,e_{0})        \nonumber\\
=&\sum\limits_{T\in \mathcal{T}_{h}}-(\nabla \phi,\nabla e_{0})_{T}+\sum\limits_{T\in \mathcal{T}_{h}}(\phi,e_{0})-\sum\limits_{T\in \mathcal{T}_{h}}\langle \nabla\phi \cdot \mathbf{n},e_{0}\rangle_{\partial T}\nonumber\\
=&0,
\end{flalign}
and
\begin{equation}
\label{3.7}
\sum\limits_{T\in \mathcal{T}_{h}}\langle \nabla\phi \cdot \mathbf{n},e_{b}\rangle_{\partial T}=\Vert e_b\Vert_{X}^{2}.
\end{equation}
Combining (\ref{3.6}) with (\ref{3.7}), we arrive at
\begin{equation}
\label{3.8}
\Vert e_{b} \Vert_{X}^{2}=\sum\limits_{T\in \mathcal{T}_{h}}(\nabla\phi,\nabla e_{0})+\sum\limits_{T\in \mathcal{T}_{h}}(\phi,e_{0})-\sum\limits_{T\in \mathcal{T}_{h}}\langle \nabla\phi \cdot \mathbf{n},e_{0}-e_{b}\rangle_{\partial T}.
\end{equation}
It follows from $(\ref{3.1})$ and $(\ref{3.8})$ that
\begin{equation}
\label{3.9}
\Vert e_{b} \Vert_{X}^{2}=(\nabla_{w}e_{h},\nabla_{w}Q_{h}\phi)+\ell(\phi,e_{h})+(\phi,e_{0}).
\end{equation}
From Lemma $\ref{lemma3.1}$, we have
\begin{equation}
\label{3.10}
(\nabla_{w}e_{h},\nabla_{w}Q_{h}\phi)=\ell(u,Q_{h}\phi)+s(Q_{h}\phi,Q_{h}u)-s(e_{h},Q_{h}\phi)-(e_{0},Q_{0}\phi).
\end{equation}
Combining $(\ref{3.9})$ with $(\ref{3.10})$, we obtain
\begin{equation}
\label{3.11}
\Vert e_{b} \Vert_{X}^{2}=\ell(u,Q_{h}\phi)+s(Q_{h}u,Q_{h}\phi)+\ell(\phi,e_{h})-s(e_{h},Q_{h}\phi).
\end{equation}
Let's analyze the four terms appearing on the right-hand side of equation $(\ref{3.11})$. We can get the following estimate by applying the trace inequality and the projection inequality
\begin{flalign}
\label{3.12}
\vert \ell(u,Q_{h}\phi)\vert=&\left\vert \sum\limits_{T\in \mathcal{T}_{h}}\langle (\nabla u-\mathbb{Q}_{h}\nabla u)\cdot \mathbf{n},Q_{0}\phi-Q_{b}\phi\rangle_{\partial T}\right\vert    \nonumber\\
\leq&\left\vert \sum\limits_{T\in \mathcal{T}_{h}}\langle (\nabla u-\mathbb{Q}_{h}\nabla u)\cdot \mathbf{n},Q_{0}\phi-\phi\rangle_{\partial T}\right\vert\nonumber\\
\leq&\sum\limits_{T\in \mathcal{T}_{h}}\Vert \nabla u-\mathbb{Q}_h\nabla u\Vert_{\partial T}\Vert Q_0\phi-\phi \Vert_{\partial T} \nonumber\\
\leq& \left(\sum\limits_{T\in \mathcal{T}_{h}}\Vert \nabla u-\mathbb{Q}_h\nabla u\Vert^2_{\partial T}\right)^{\frac{1}{2}}\left(\sum\limits_{T\in \mathcal{T}_{h}}\Vert Q_0\phi-\phi \Vert^2_{\partial T}\right)^{\frac{1}{2}} \nonumber\\
\leq&Ch^{k+\frac{r}{2}}\Vert u\Vert_{k+1}\Vert \phi \Vert_{1+r/2}.
\end{flalign}
Similarly, we have
\begin{flalign}
\label{3.17}
\vert s(Q_hu,Q_h\phi)\vert&=\gamma(h)\sum\limits_{T\in \mathcal{T}_{h}}h_T^{-1}\vert \langle Q_0u-Q_bu,Q_0\phi-Q_b\phi \rangle \vert  \nonumber\\
&\leq \gamma(h)\sum\limits_{T\in \mathcal{T}_{h}}h_T^{-1}\Vert Q_0u-u\Vert_{\partial T} \Vert Q_0\phi-\phi \Vert_{\partial T}  \nonumber\\
&\leq \gamma(h)\left(\sum\limits_{T\in \mathcal{T}_{h}}h_T^{-1}\Vert Q_0u-u\Vert^2_{\partial T}\right)^{\frac{1}{2}} \left(\sum\limits_{T\in \mathcal{T}_{h}}h_T^{-1}\Vert Q_0\phi-\phi \Vert^2_{\partial T}\right)^{\frac{1}{2}}  \nonumber\\
&\leq C\gamma(h)h^{k+\frac{r}{2}}\Vert u\Vert_{k+1}\Vert \phi \Vert_{1+r/2}.
\end{flalign}
It follows from Lemmas $\ref{lemma2.3}$ and $\ref{lemma3.2}$ that
$$\Vert e_h\Vert_V \leq C\gamma(h)^{-1}h^k\Vert u\Vert_{k+1}.$$
From Lemma $\ref{lemma3.2}$, the following estimates hold
\begin{flalign}
\label{3.18}
\vert \ell(\phi,e_h)\vert\leq Ch^{\frac{r}{2}}\Vert \phi \Vert_{1+r/2}\Vert e_h\Vert_V \leq C\gamma(h)^{-1}h^{k+\frac{r}{2}}\Vert u\Vert_{k+1}\Vert \phi \Vert_{1+r/2},
\end{flalign}
\begin{flalign}
\label{3.19}
\vert s(e_h,Q_h\phi)\vert\leq C\gamma(h)h^{\frac{r}{2}}\Vert \phi \Vert_{1+r/2}\Vert e_h\Vert_V\leq Ch^{k+\frac{r}{2}}\Vert u\Vert_{k+1}\Vert \phi \Vert_{1+r/2}.
\end{flalign}
To sum up, combining $(\ref{3.11})$-$(\ref{3.19})$ leads to the following estimate
\begin{equation}
\label{3.20}
\Vert e_{b}\Vert_{X}^{2}\leq C\gamma(h)^{-1}h^{k+\frac{r}{2}}\Vert u\Vert_{k+1}\Vert \phi \Vert_{1+r/2}.
\end{equation}
In addition, the solution $\phi$ of the dual problem $(\ref{3.5})$ satisfies $\Vert \phi\Vert_{1+r/2}\leq C\Vert e_{b}\Vert_{X}$, so we have
\begin{equation}
\label{3.21}
\Vert e_{b}\Vert_{X}\leq C\gamma(h)^{-1}h^{k+\frac{r}{2}}\Vert u\Vert_{k+1}.
\end{equation}
Combining Lemma $\ref{lemma2.5}$ with $(\ref{3.21})$ yields
$$\Vert u-u_{h}\Vert_{X}\leq \Vert u-Q_{h}u\Vert_{X}+\Vert Q_{h}u-u_{h}\Vert_{X} \leq C\gamma(h)^{-1}h^{k+\frac{r}{2}}\Vert u\Vert_{k+1}.$$
Then we complete the proof.
\end{proof}

\section{Guaranteed lower bounds}
In this section, we give a new numerical scheme and prove the GLB theorem.

Let $\mathcal{T}_h$ be a regular triangulation of the domain $\Omega \subset \mathbb{R}^d~(d=2, 3)$ into simplices, $\alpha$ be a global parameter that satisfies $\alpha>0$.
Next we give the another WG algorithm of (\ref{2.1}).

\begin{algorithm1}
Find $(\lambda_h, u_h)\in\Real\times V_h$ such that $\|u_h\|_b=1$ and
\begin{eqnarray}\label{4.1}
a_w(u_h,v_h)=\lambda_h b_w(u_h,v_h),\quad\forall v_h\in V_h,
\end{eqnarray}
\end{algorithm1}
where
\begin{align*}
&a_w(v_{h},w_{h})=(\nabla_w v_{h},\nabla_w w_{h})+(v_{0},w_{0})+s(v_{h},w_{h}),\\
&b_w(v_{h},w_{h})=\langle v_{b},w_{b}\rangle_{L^2(\partial \Omega)},\\
&s(v_{h},w_{h})=\frac{\alpha}{d+1}\sum\limits_{T\in \mathcal{T}_h}\sum\limits_{e\in \mathcal{E}(T)}h_{T}^{-2}\vert e\vert^{-1}\vert T\vert \langle v_{0}-v_{b},w_{0}-w_{b}\rangle_{L^2(e)}.
\end{align*}

Positive constants $\delta$ and $\Lambda$ are related to the maximal diameter $h_{\max}$ in the shape-regular triangulation $\mathcal{T}_h$, then we have the following two fundamental estimates
\begin{align}
&\Vert (I-Q_b)f\Vert^2_{X} \leq \delta \Vert (I-\mathbb{Q}_h)\nabla f\Vert^2, \quad \forall f \in H^1(\Omega), \tag{A}\label{A} \\
&s(Q_{h}f,Q_{h}f) \leq \alpha \Lambda \Vert (I-\mathbb{Q}_h)\nabla f\Vert^2, \quad \forall f \in H^1(\Omega).\tag{B}\label{B}
\end{align}

\begin{proof}
From  \cite[Proposition 4.3]{2020A}, we have
\begin{align*}
    \Vert (I-Q_b)f\Vert_X^2 &\leq \Vert (I-Q_0)f\Vert_X^2\\
    &\leq \delta_1\Vert \nabla(I-Q_0)f\Vert^2\\
    &\leq C_{st}^2\delta_1\Vert(I-\mathbb{Q}_h)\nabla f\Vert^2\\
    &=\delta \Vert(I-\mathbb{Q}_h)\nabla f\Vert^2.
\end{align*}
It follows from \cite[Theorem 4.1]{2020A} that
$$s(Q_{h}f,Q_{h}f) \leq \alpha \Lambda \Vert (I-\mathbb{Q}_h)\nabla f\Vert^2.$$
Then we complete the proof.
\end{proof}

\begin{remark}
In the above proof, $\Lambda, \,C_{st}, \,\delta_1$ are positive constants, and when $k=1$, $\Lambda, \,\delta_1$ have explicit expressions. The details can be seen in \cite[Theorem 3.1, Proposition 4.3]{2020A}.
\end{remark}

Let $\lambda$ be the $j$-th exact eigenvalue satisfying (\ref{2.2}), $\lambda_h$ be the $j$-th discrete eigenvalue satisfying (\ref{4.1}).
\begin{theorem}
(GLB). If $\delta$, $\Lambda$ and $\alpha>0$ satisfy either $(1)$ $\delta\lambda+\alpha\Lambda \leq1$ or $(2)$ $\delta\lambda_h+\alpha\Lambda \leq1$, then $\lambda_h \leq \lambda.$
\end{theorem}
\begin{proof}
The proof of this theorem is similar to that of \cite[Theorem 5.1]{2020A}.\\
\textit{Step one.} If $\delta \lambda \geq 1$, the first condition (1) is impossible. So when the second condition (2) holds, we can obtain $\lambda_h \leq \lambda$. Therefore, for the remaining steps of this proof, we only consider case $\delta \lambda<1$.\\
\textit{Step two.}~Let $\phi_1,\cdots,\phi_j$ be the first $j\in \mathbb{N}$ exact eigenfunctions of the Steklov eigenvalue problem (\ref{2.2}). Let $\lambda=\lambda_j$ be the $j$-th exact eigenvalue. Next, we prove that $Q_b\phi_1,\cdots,Q_b\phi_j$ are linearly independent in $P_k(\mathcal{E}_h)$. We assume that they are linearly dependent, then suppose that $\psi \in {\rm{span}}\{ \phi_1,\cdots,\phi_j\}$ satisfies $\Vert \psi \Vert_{X} = 1$ and $Q_b\psi = 0$ on $\partial \Omega$.
It follows from (\ref{A}) that
\begin{align*}
    1=\Vert \psi \Vert_{X}^2=\Vert (I-Q_b)\psi\Vert_{X}^2 \leq \delta \Vert (I-\mathbb{Q}_h)\nabla \psi\Vert^2.
\end{align*}
According to the orthogonality of the projection operator, we have
\begin{equation}
\label{4.2}
     \Vert{\psi}\Vert_1^2=\Vert \psi \Vert^2+\Vert \nabla \psi\Vert^2=\Vert \psi \Vert^2+\Vert \mathbb{Q}_h\nabla \psi\Vert^2+\Vert (I-\mathbb{Q}_h)\nabla \psi\Vert^2.
\end{equation}
Thus, we can obtain
$$1\leq \delta \Vert{\psi}\Vert_1^2.$$
On the other hand, it follows from the minimum-maximum principle on the exact eigenvalues that
$$\Vert{\psi}\Vert_1^2 \leq \lambda \Vert \psi\Vert_{X}^2=\lambda.$$
From the above two inequalities, we have $1 \leq \delta \lambda$.
This contradicts $\delta \lambda < 1$.\\
\textit{Step three.} Step two proves that $Q_b\phi_1,\cdots,Q_b\phi_j$ are linearly independent in $P_k(\mathcal{E}_h)$. So $U_h$=span$\{Q_h\phi_1,\cdots,Q_h\phi_j\}$ is a linear subspace of $V_h$ with dimension $j$. Let $\mathcal{U}(j)$ denote the set of
all subspaces of $V_h$ with dimension $j$. So the minimum-maximum principle characterizes the $j$-th discrete eigenvalue $\lambda_h$ as
\begin{equation}
\label{4.3}
    \lambda_h=\mathop{\min}\limits_{U \in \mathcal{U}(j)}\mathop{\max}\limits_{v_h \in U\setminus \{0\}}\frac{a_w(v_h,v_h)}{b_w(v_h,v_h)}.
\end{equation}
Therefore, $\lambda_h$ is a lower bound for the maximal quotient $a_h(v_h,v_h)/b_h(v_h,v_h)$ among all nonzero $v_h \in U_h$. In the finite dimensional space, the maximum can be obtained, so there exists $\phi \in$ span $\{\phi_1,\cdots,\phi_j\}$ that satisfies $\Vert \phi \Vert_{X}=1$ and
\begin{equation}
\label{4.4}
    \lambda_h b_w(Q_h\phi,Q_h\phi) \leq a_w(Q_h\phi,Q_h\phi).
\end{equation}
In addition, it follows from the minimum-maximum principle on the exact eigenvalues that $\Vert \phi \Vert_1^2 \leq \lambda$.\\
\textit{Step four.} We estimate the lower bound of $b_w(Q_h\phi,Q_h\phi)$. Recall that $\Vert \phi \Vert_{X}=1$, we have
\begin{equation*}
    b_w(Q_h\phi,Q_h\phi)=b_w(Q_b\phi,Q_b\phi)=\Vert Q_b\phi \Vert_{X}^2=1-\Vert(I-Q_b)\phi \Vert_{X}^2.
\end{equation*}
Combining the above equation with (\ref{A}), we arrive at
\begin{equation}
\label{4.5}
    1-\delta\Vert (I-\mathbb{Q}_h)\nabla \phi\Vert^2 \leq b_w(Q_h\phi,Q_h\phi).
\end{equation}
\textit{Step five.} We estimate the upper bound of $a_w(Q_h\phi,Q_h\phi)$.
It follows from Lemma \ref{lemma2.2} that $\mathbb{Q}_h\nabla \phi=\nabla_wQ_h\phi$. This leads to
\begin{equation}
\label{4.6}
    a_w(Q_h\phi,Q_h\phi)=\Vert \mathbb{Q}_h\nabla \phi\Vert^2+\Vert Q_0\phi \Vert^2+s(Q_h\phi,Q_h\phi).
\end{equation}
It follows from (\ref{4.6}), (\ref{B}), (\ref{4.2}) and $\Vert \phi\Vert_1^2\leq 1$ that
\begin{equation}
\label{4.7}
    a_w(Q_h\phi,Q_h\phi)+\Vert (I-\mathbb{Q}_h)\nabla \phi\Vert^2  \leq \lambda +\alpha\Lambda\Vert (I-\mathbb{Q}_h)\nabla \phi\Vert^2.
\end{equation}
\textit{Step six.} Combining (\ref{4.4}) with (\ref{4.5}) and (\ref{4.7}), we have
\begin{equation}
\label{4.8}
    (1-\alpha\Lambda-\delta\lambda_h)\Vert(I-\mathbb{Q}_h)\nabla \phi \Vert^2 \leq \lambda-\lambda_h.
\end{equation}
In case (2), the left-hand side of (\ref{4.8}) is nonnegative. This shows that $\lambda-\lambda_h \geq 0$. In case (1), from (\ref{4.8}), we can obtain
\begin{equation}
    \label{4.9}
    \delta(\lambda-\lambda_h)\Vert (I-\mathbb{Q}_h)\nabla\phi\Vert^2 \leq \lambda-\lambda_h.
\end{equation}
Combining (\ref{4.2}) with $\Vert \phi \Vert_1^2\leq \lambda$, we have $\delta\Vert(I-\mathbb{Q}_h)\nabla \phi \Vert^2 \leq \delta \lambda <1$. This shows that $\lambda-\lambda_h$ must be nonnegative.
Then we complete the proof.
\end{proof}

\section{Numerical experiments}In this section, we present some numerical experiments
to demonstrate the accuracy and lower bound property of the WG scheme $(\ref{2.4})$.
\subsection{Square domain}
In this example, we consider the Steklov eigenvalue  problem $(\ref{2.1})$ on the square domain $\Omega=(0,1)^2$. The parameter $\gamma(h)$ in $(\ref{2.4})$ is selected as $h^{0.1}$. The degree $k$ of the polynomial space is selected as 1 or 2, respectively. As we lack access to the exact eigenvalues, we have utilized the following values as the first four exact eigenvalues in our numerical assessments: $\lambda_{1}=0.2400790854320629$, $\lambda_{2}=1.492303134033900$, $\lambda_{3}=1.492303134115401$, and $\lambda_{4}=2.082647053961881$. The outcomes of our numerical experiments can be found in Tables $\ref{table5.1}$ - $\ref{table5.2}$, while visual representations of the first four eigenfunctions can be seen in Figs $\ref{figure5.1}$ - $\ref{figure5.4}$.

\begin{table}[htbp]
\centering
\caption{$\Omega=(0,1)^{2}$, $k=1$, $\gamma(h)=h^{0.1}$.}
 \label{table5.1}
\renewcommand\arraystretch{1}
\begin{tabular}{|c | c |c | c |c |}
\hline
$h$ & 1/8 & 1/16 & 1/32 & 1/64   \\ [0ex]
 \hline\hline
 $\lambda_{1}-\lambda_{1,h}$ & \makecell[c]{8.0705e-4} &   \makecell[c]{2.1839e-4}   &  \makecell[c]{5.8965e-5}   &  \makecell[c]{1.5905e-5} \\
 \hline
 order &   & \makecell[c]{1.8858}   &  \makecell[c]{1.8890}    & \makecell[c]{1.8903}\\
 \hline
 $\lambda_{2}-\lambda_{2,h}$ & \makecell[c]{1.6157e-02}  &   \makecell[c]{2.9073e-03}   &  \makecell[c]{6.4444e-04}  &   \makecell[c]{1.5391e-04} \\
 \hline
 order &  & \makecell[c]{2.4744}   &  \makecell[c]{2.1736}    & \makecell[c]{2.0659}\\
 \hline
 $\lambda_{3}-\lambda_{3,h}$ &  \makecell[c]{1.4404e-02}  &     \makecell[c]{2.7983e-03}   &  \makecell[c]{6.3785e-04}    & \makecell[c]{1.5357e-04}\\
 \hline
 order &   & \makecell[c]{2.3638}  &   \makecell[c]{2.1333}  &  \makecell[c]{2.0543}\\
 \hline
 $\lambda_{4}-\lambda_{4,h}$ & \makecell[c]{7.3939e-02}  &   \makecell[c]{1.0259e-02}   &  \makecell[c]{2.1026e-03}   &  \makecell[c]{4.8367e-04}
\\
 \hline
 order &   & \makecell[c]{2.8494}   & \makecell[c]{2.2867} &   \makecell[c]{2.1201}\\
 \hline
\end{tabular}
\end{table}

\begin{table}[htbp]
\centering
\caption{$\Omega=(0,1)^{2}$, $k=2$, $\gamma(h)=h^{0.1}$.}
 \label{table5.2}
 \renewcommand\arraystretch{1}
\begin{tabular}{|c | c |c | c |c |}
 \hline
$h$ & 1/8 & 1/16 & 1/32 & 1/64  \\ [0ex]
 \hline\hline
 $\lambda_{1}-\lambda_{1,h}$ & \makecell[c]{1.8201e-08}    & \makecell[c]{1.1757e-09}  &   \makecell[c]{8.0361e-11}  &   \makecell[c]{4.8654e-12} \\
 \hline
 order &   & \makecell[c]{3.9524}   &  \makecell[c]{3.8709}   & \makecell[c]{4.0459}\\
 \hline
 $\lambda_{2}-\lambda_{2,h}$ & \makecell[c]{3.5680e-05}    & \makecell[c]{2.0584e-06}  &   \makecell[c]{1.2589e-07}  &   \makecell[c]{7.4144e-09} \\
 \hline
 order &  & \makecell[c]{4.1155}  &   \makecell[c]{4.0312}   & \makecell[c]{4.0858}\\
 \hline
 $\lambda_{3}-\lambda_{3,h}$ &  \makecell[c]{2.9710e-05}   &  \makecell[c]{1.7377e-06}   &  \makecell[c]{1.0639e-07}    & \makecell[c]{6.2445e-09}\\
 \hline
 order &   & \makecell[c]{4.0957}   &  \makecell[c]{4.0297}   & \makecell[c]{4.0907}\\
 \hline
 $\lambda_{4}-\lambda_{4,h}$ & \makecell[c]{3.9009e-06}   & \makecell[c]{2.4281e-07}  &   \makecell[c]{1.5940e-08}  &   \makecell[c]{9.9535e-10}\\
 \hline
 order &   & \makecell[c]{4.0059}  &  \makecell[c]{3.9291}   &  \makecell[c]{4.0013}\\
 \hline
\end{tabular}
\end{table}

\begin{figure}[htbp]
\centering
\begin{minipage}[t]{0.48\textwidth}
\centering
\includegraphics[width=6cm]{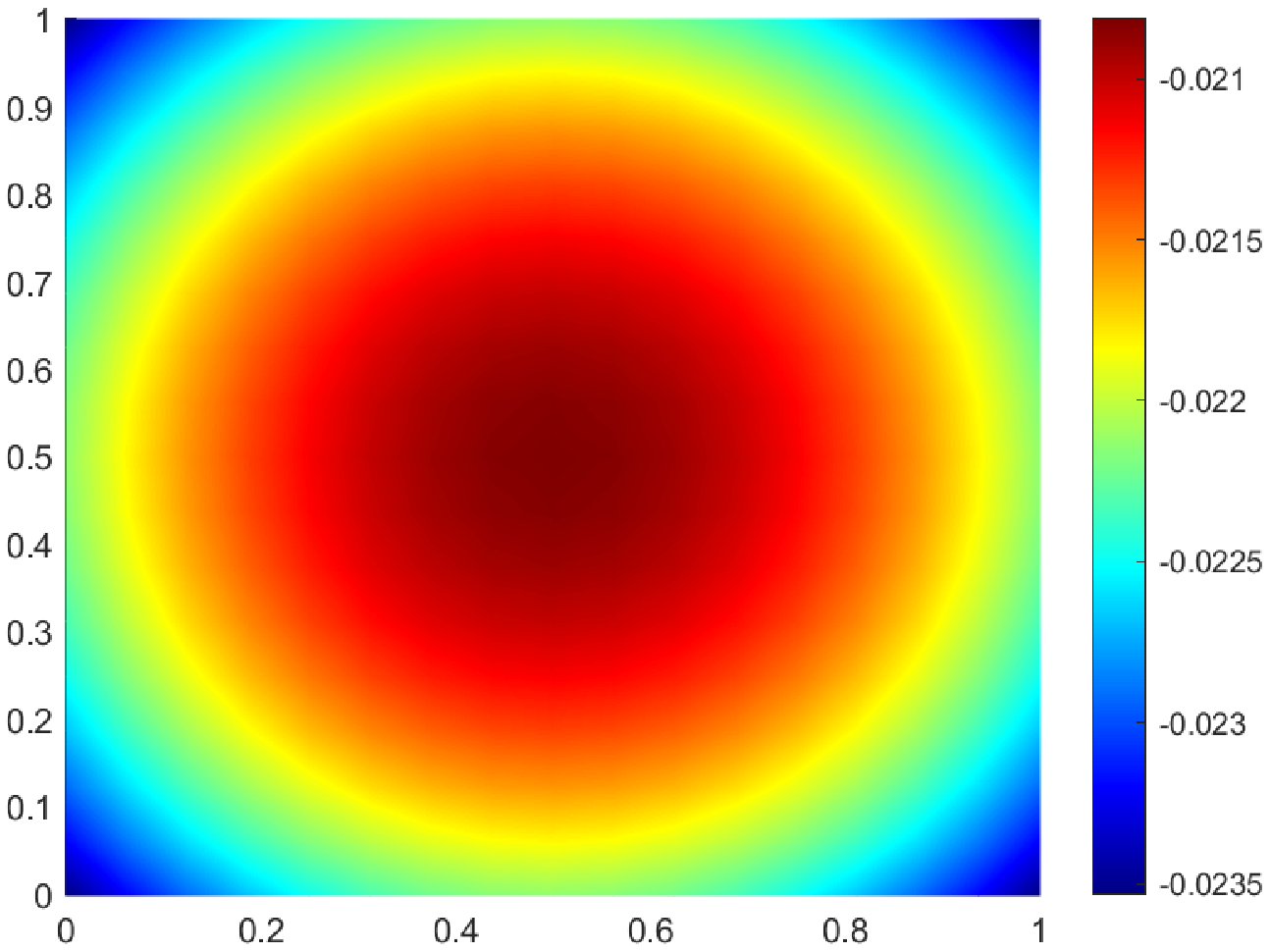}
\caption{$\Omega=(0,1)^2$, $k=1$, $\gamma(h)=h^{0.1}$, $h=1/16$, the first eigenfunction.}
\label{figure5.1}
\end{minipage}
\begin{minipage}[t]{0.48\textwidth}
\centering
\includegraphics[width=6cm]{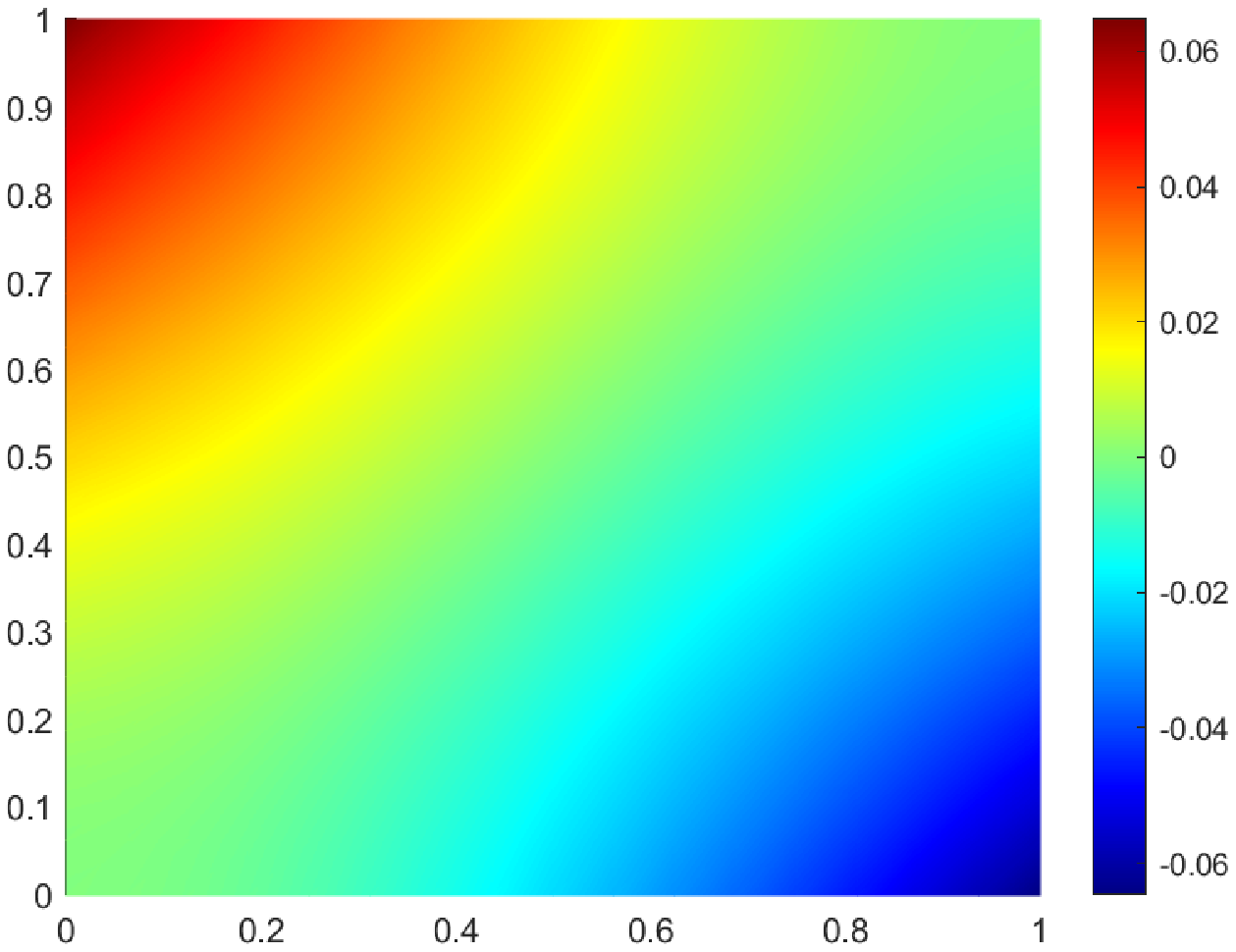}
\caption{$\Omega=(0,1)^2$, $k=1$, $\gamma(h)=h^{0.1}$, $h=1/16$, the second eigenfunction.}
\label{figure5.2}
\end{minipage}
\begin{minipage}[t]{0.48\textwidth}
\centering
\includegraphics[width=6cm]{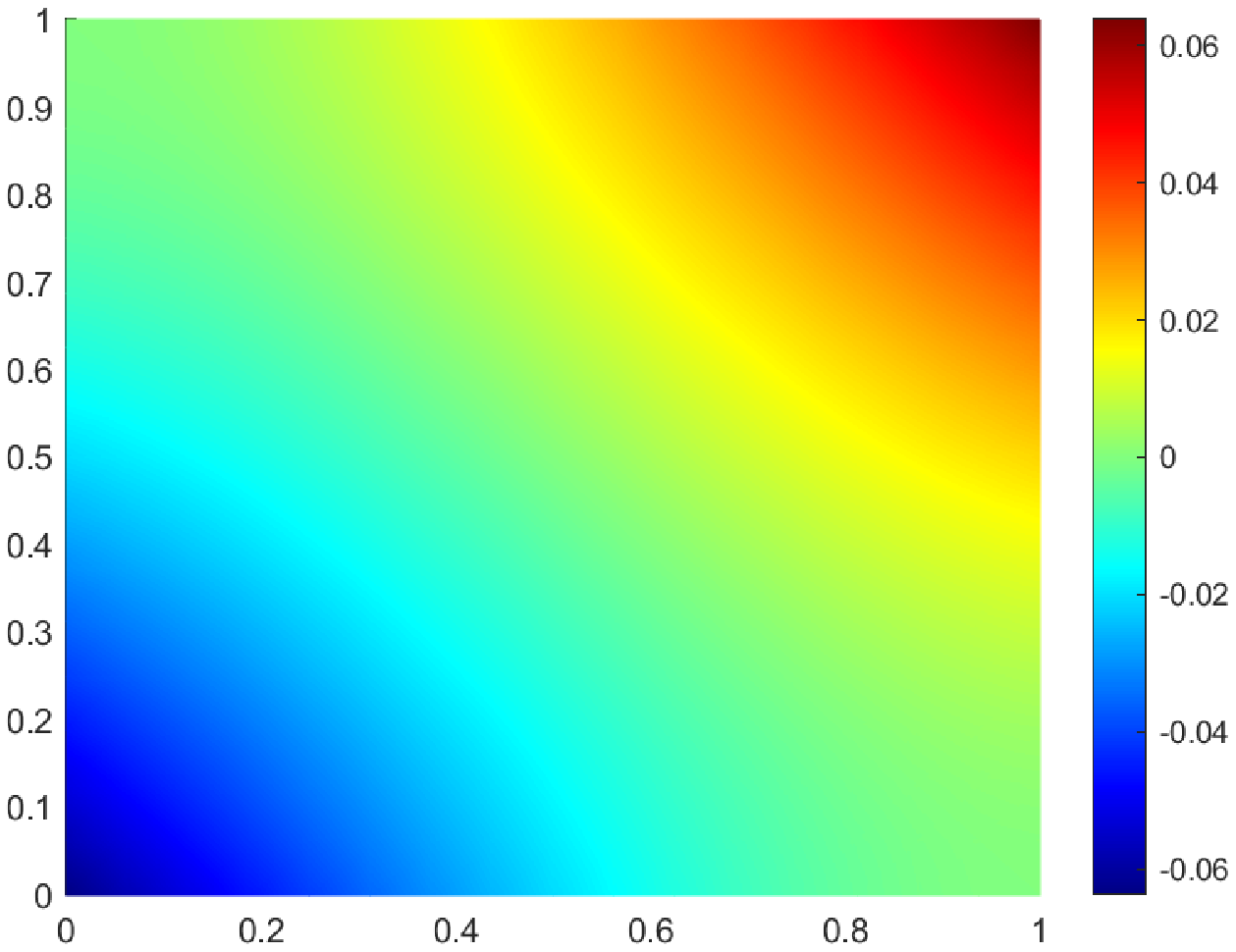}
\caption{$\Omega=(0,1)^2$, $k=1$, $\gamma(h)=h^{0.1}$, $h=1/16$, the third eigenfunction.}
\label{figure5.3}
\end{minipage}
\begin{minipage}[t]{0.48\textwidth}
\centering
\includegraphics[width=6cm]{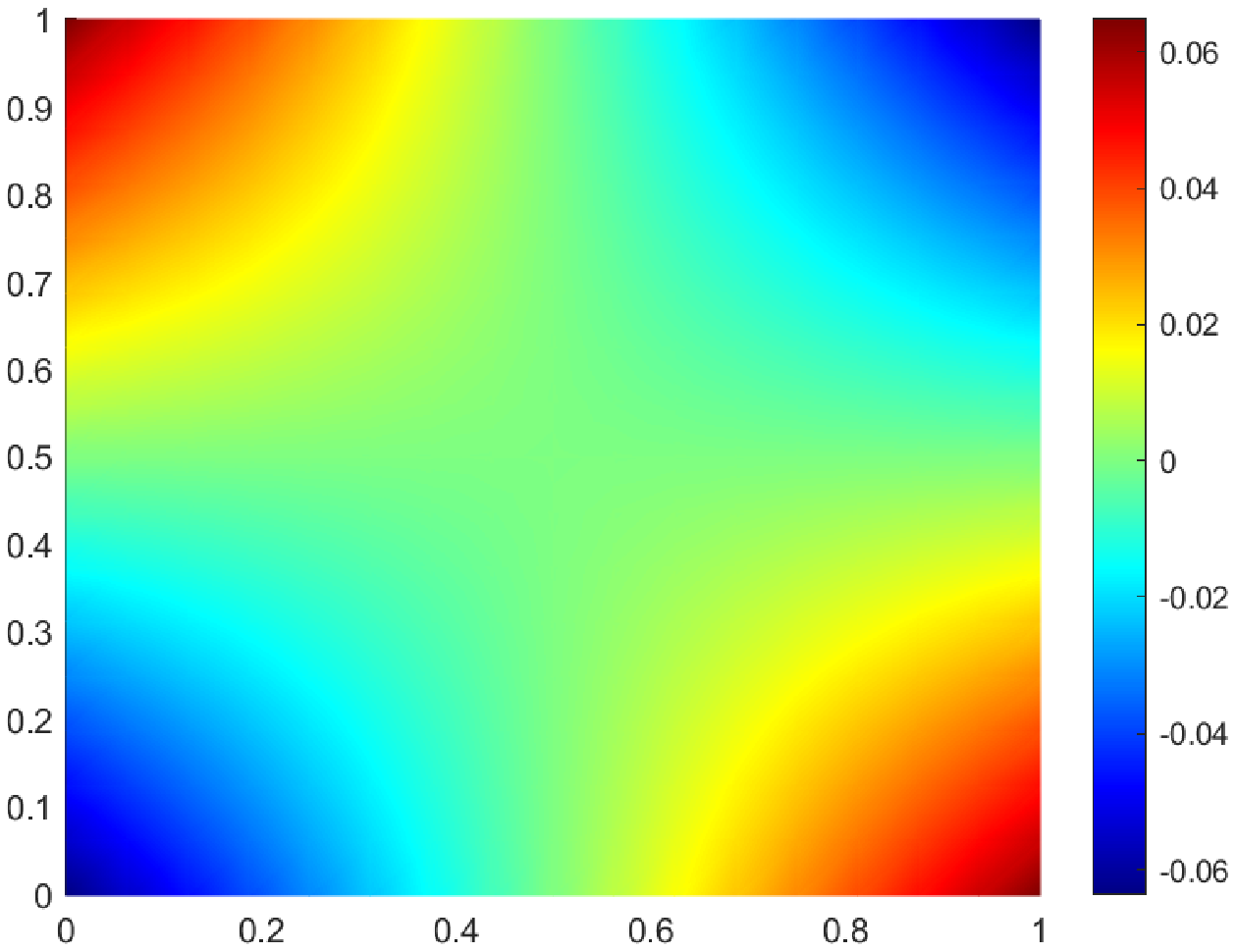}
\caption{$\Omega=(0,1)^2$, $k=1$, $\gamma(h)=h^{0.1}$, $h=1/16$, the fourth eigenfunction.}
\label{figure5.4}
\end{minipage}
\end{figure}

The data in Tables $\ref{table5.1}$ - $\ref{table5.2}$ indicates that the eigenvalues obtained by the
WG method serve as asymptotic lower bound approximations for the exact eigenvalues of the Steklov eigenvalue problem $(\ref{2.1})$ on the square domain $\Omega=(0,1)^2$. This aligns with the findings from our theoretical analysis presented in Theorem $\ref{theorem2.15}$. Specifically, for values of $k$ equal to 1 and 2, we anticipate the convergence orders are $\mathcal{O}(h^{1.8})$ and $\mathcal{O}(h^{3.8})$ from Theorem $\ref{theorem2.9}$, respectively. It's worth noting that for the square domain, Tables $\ref{table5.1}$ and $\ref{table5.2}$ indicate that all the observed convergence orders are optimal for both $k=1$ and $k=2$.

\subsection{L-shaped domain}
We select different settings on the L-shaped domain $\Omega=(0,1)\times (0,1)\setminus[\frac{1}{2},1]\times [\frac{1}{2},1]$. The parameter $\gamma(h)$ in $(\ref{2.4})$ is selected as $h^{0.1}$, $h^{0.2}$ or $-1/\log(h)$, respectively. The degree $k$ of the polynomial space is selected as 1 or 5, respectively. The numerical results can be found in Tables $\ref{table5.5}$ - $\ref{table5.7}$, and the figures of the first four eigenfunctions are presented in Figs $\ref{figure5.5}$ - $\ref{figure5.8}$.

\begin{table}[htbp]
\centering
 \caption{$\Omega=(0,1)\times (0,1)\setminus[\frac{1}{2},1]\times [\frac{1}{2},1]$, $k=5$, $\gamma(h)=h^{0.1}$.}
  \label{table5.5}
 \renewcommand\arraystretch{1}
\begin{tabular}{|c | c |c | c |c |c|}
 \hline
$h$ & 1/8 & 1/16 & 1/32 & 1/64 & Trend \\[0ex] 
 \hline\hline
 $\lambda_{1,h}$ &  \makecell[c]{0.1829642327}   &  \makecell[c]{0.1829642362}    & \makecell[c]{0.1829642369}  &  \makecell[c]{0.1829642374} & $\nearrow$\\
 \hline
 $\lambda_{2,h}$ & \makecell[c]{0.8931819640}  &   \makecell[c]{0.8934620312}   &  \makecell[c]{0.8935729025}  &  \makecell[c]{0.8936168669} & $\nearrow$\\
 \hline
  $\lambda_{3,h}$ & \makecell[c]{1.688598908} &   \makecell[c]{1.688600231}   &  \makecell[c]{1.688600439}  &   \makecell[c]{1.688600472} & $\nearrow$\\
 \hline
  $\lambda_{4,h}$ & \makecell[c]{3.217859767}  &   \makecell[c]{3.217859784}    & \makecell[c]{3.217859787}  &   \makecell[c]{3.217859788} & $\nearrow$\\
 \hline
\end{tabular}
\end{table}

\begin{table}[htbp]
\centering
\caption{$\Omega=(0,1)\times (0,1)\setminus[\frac{1}{2},1]\times [\frac{1}{2},1]$, $k=5$, $\gamma(h)=h^{0.2}$.}
 \label{table5.6}
 \renewcommand\arraystretch{1}
\begin{tabular}{|c | c |c | c |c |c|}
 \hline
$h$ & 1/8 & 1/16 & 1/32 & 1/64 & Trend  \\ [0ex] 
 \hline\hline
 $\lambda_{1,h}$ &  \makecell[c]{0.1829642326}  &   \makecell[c]{0.1829642362}    & \makecell[c]{0.1829642369}   &  \makecell[c]{0.1829642374} & $\nearrow$\\
 \hline
 $\lambda_{2,h}$ & \makecell[c]{0.8931807937}   &  \makecell[c]{0.8934614445}    & \makecell[c]{0.8935726320}  &   \makecell[c]{0.8936167483} & $\nearrow$\\
 \hline
  $\lambda_{3,h}$ & \makecell[c]{1.688598898}  &   \makecell[c]{1.688600229}    &  \makecell[c]{1.688600438}   &  \makecell[c]{1.688600471} & $\nearrow$\\
 \hline
  $\lambda_{4,h}$ & \makecell[c]{2.779233802} &    \makecell[c]{3.217859784}   &  \makecell[c]{3.217859787}  &   \makecell[c]{3.217859788} & $\nearrow$\\
 \hline
\end{tabular}
\end{table}

\begin{table}[htbp]
\centering
\caption{$\Omega=(0,1)\times (0,1)\setminus[\frac{1}{2},1]\times [\frac{1}{2},1]$, $k=1$, $\gamma(h)=-1/\log(h)$.}
 \label{table5.7}
 \renewcommand\arraystretch{1}
\begin{tabular}{|c | c |c | c |c |c|}
 \hline
$h$ & 1/8 & 1/16 & 1/32 & 1/64 & Trend  \\[0ex] 
 \hline\hline
 $\lambda_{1,h}$ &  \makecell[c]{0.1819814683}  &   \makecell[c]{0.1826153945}    & \makecell[c]{0.1828514041}   &  \makecell[c]{0.1829296478} & $\nearrow$\\
 \hline
 $\lambda_{2,h}$ & \makecell[c]{0.8640089893}   &  \makecell[c]{0.8828364959}    & \makecell[c]{0.8896347423}  &   \makecell[c]{0.8921265258} & $\nearrow$\\
 \hline
  $\lambda_{3,h}$ & \makecell[c]{1.627317546}  &   \makecell[c]{1.679807357}    &  \makecell[c]{1.686909980}   &  \makecell[c]{1.688227347} & $\nearrow$\\
 \hline
  $\lambda_{4,h}$ & \makecell[c]{1.740316947} &    \makecell[c]{2.508500102}   &  \makecell[c]{3.199957769}  &   \makecell[c]{3.215335736} & $\nearrow$\\
 \hline
\end{tabular}
\end{table}

\begin{figure}[htbp]
\centering
\begin{minipage}[t]{0.48\textwidth}
\centering
\includegraphics[width=6cm]{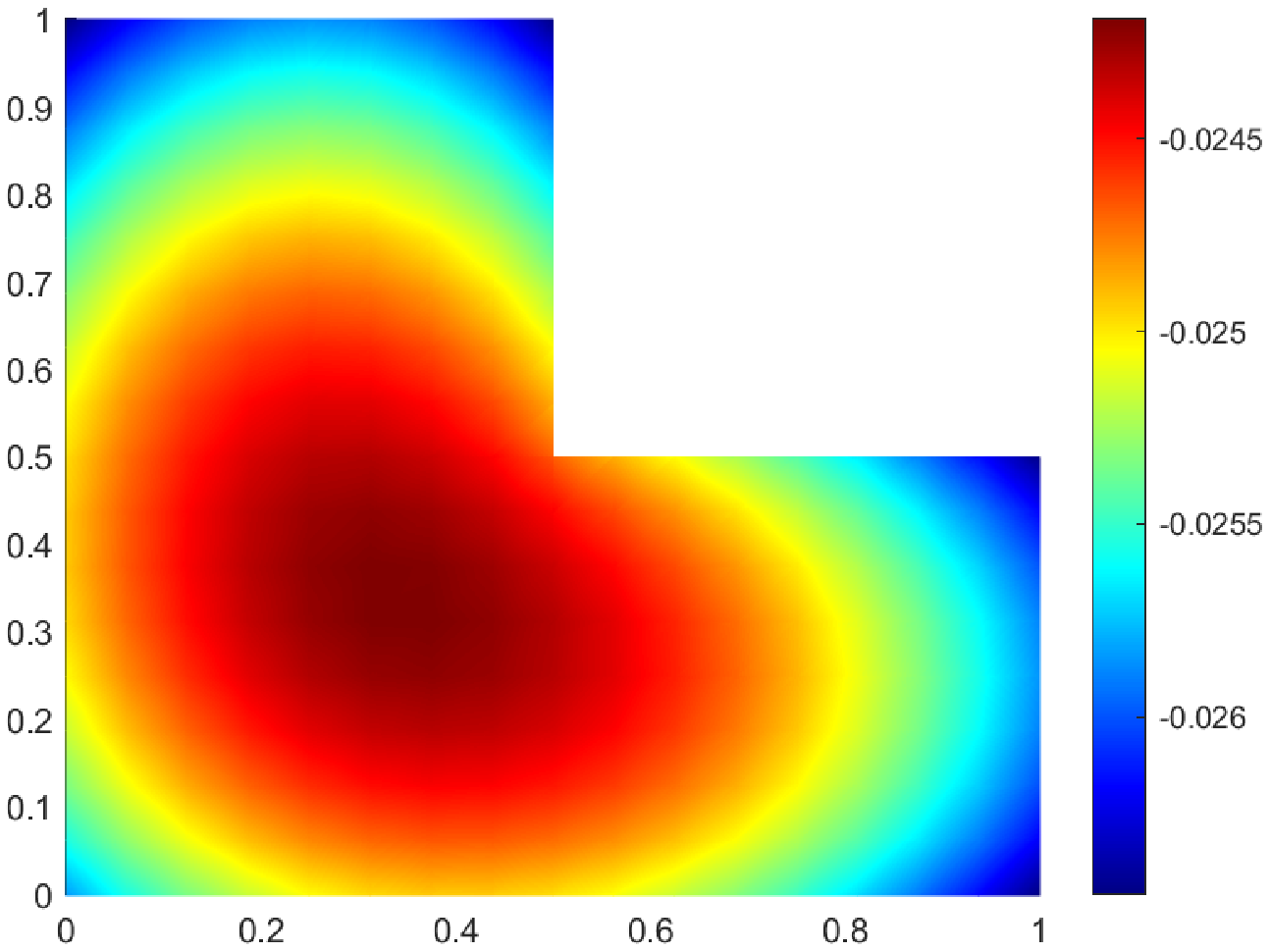}
\caption{$\Omega=(0,1)\times (0,1)\setminus[\frac{1}{2},1]\times [\frac{1}{2},1]$, $k=1$, $\gamma(h)=h^{0.1}$, $h=1/16$, the first eigenfunction.}
\label{figure5.5}
\end{minipage}
\begin{minipage}[t]{0.48\textwidth}
\centering
\includegraphics[width=6cm]{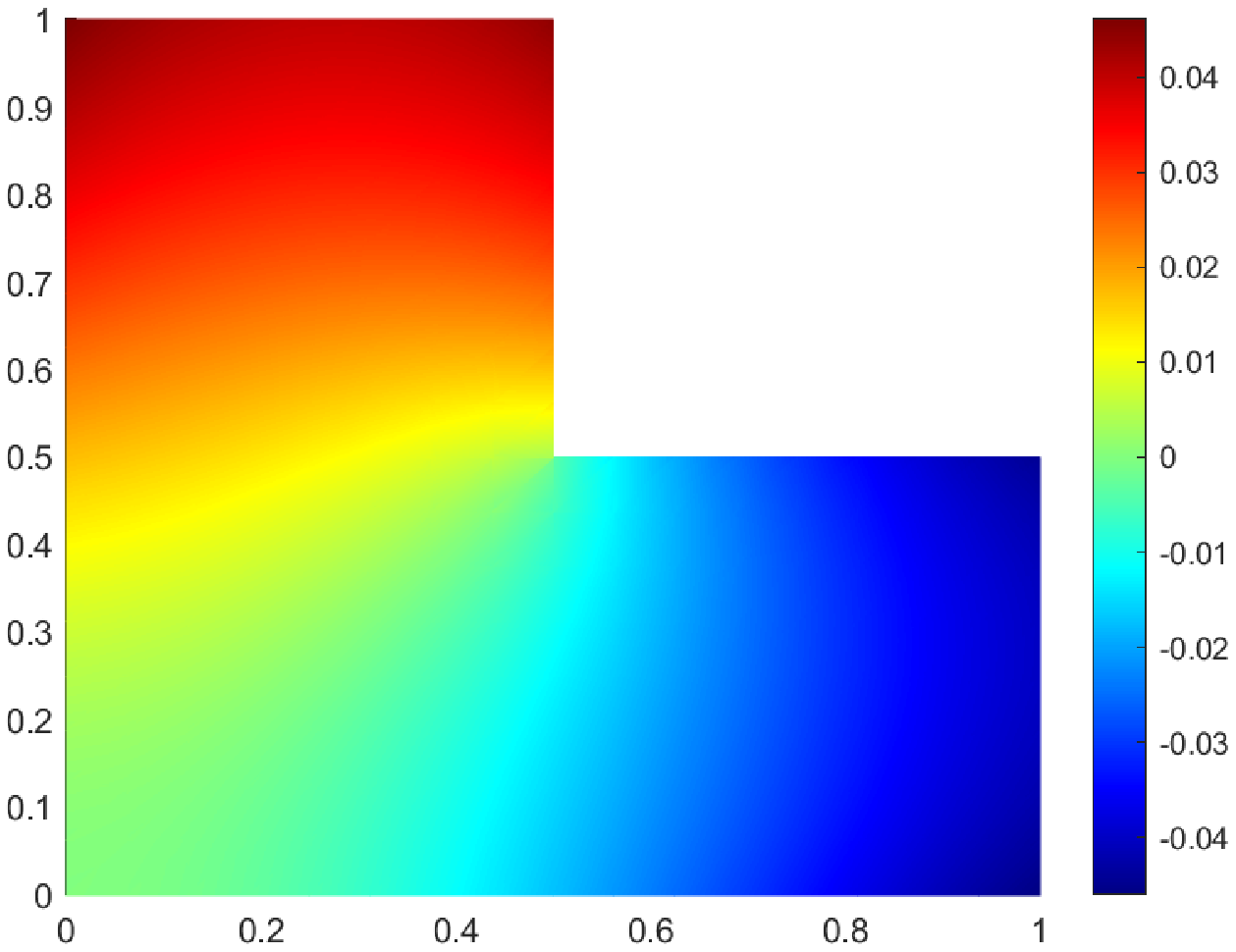}
\caption{$\Omega=(0,1)\times (0,1)\setminus[\frac{1}{2},1]\times [\frac{1}{2},1]$, $k=1$, $\gamma(h)=h^{0.1}$, $h=1/16$, the second eigenfunction.}
\label{figure5.6}
\end{minipage}
\begin{minipage}[t]{0.48\textwidth}
\centering
\includegraphics[width=6cm]{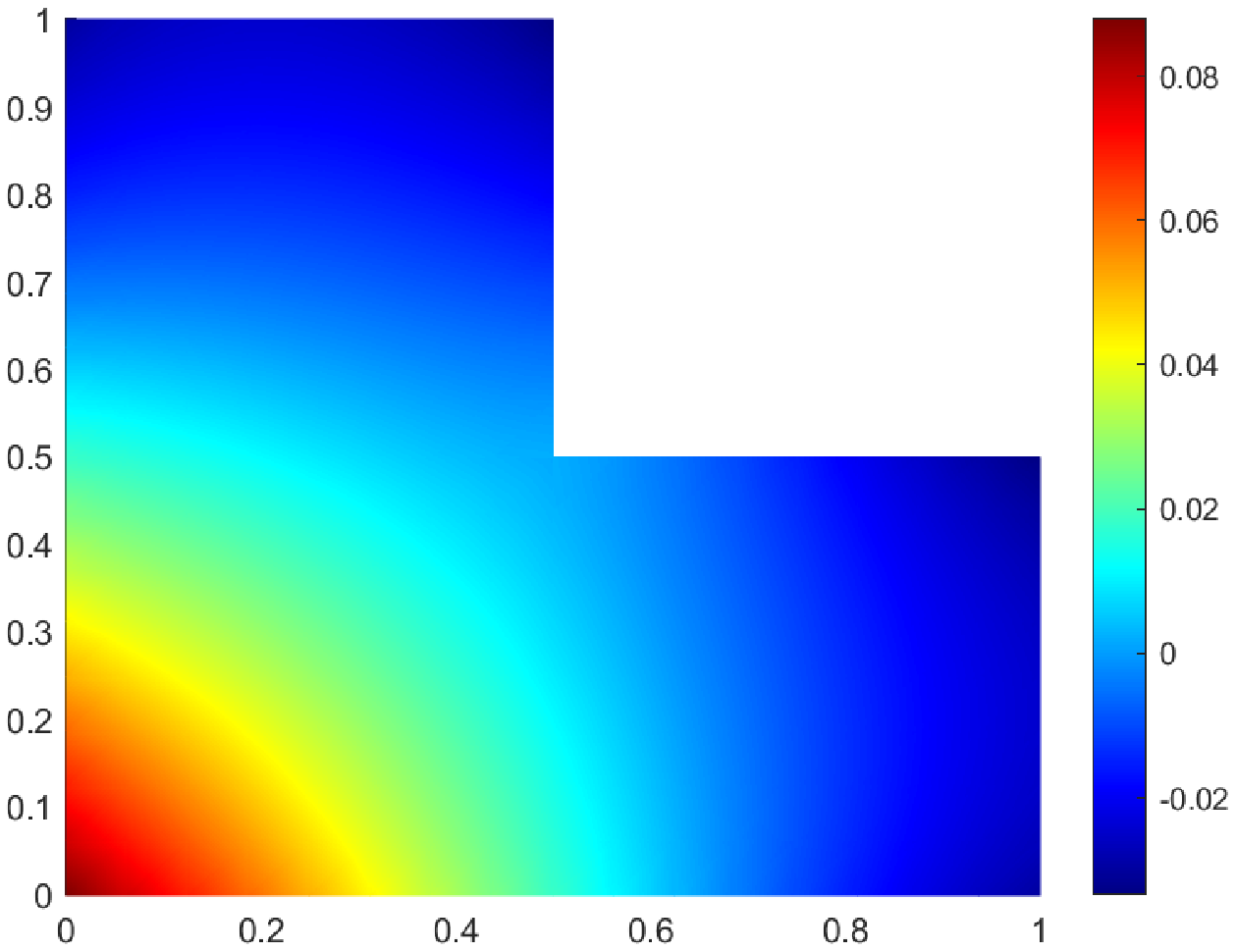}
\caption{$\Omega=(0,1)\times (0,1)\setminus[\frac{1}{2},1]\times [\frac{1}{2},1]$, $k=1$, $\gamma(h)=h^{0.1}$, $h=1/16$, the third eigenfunction.}
\label{figure5.7}
\end{minipage}
\begin{minipage}[t]{0.48\textwidth}
\centering
\includegraphics[width=6cm]{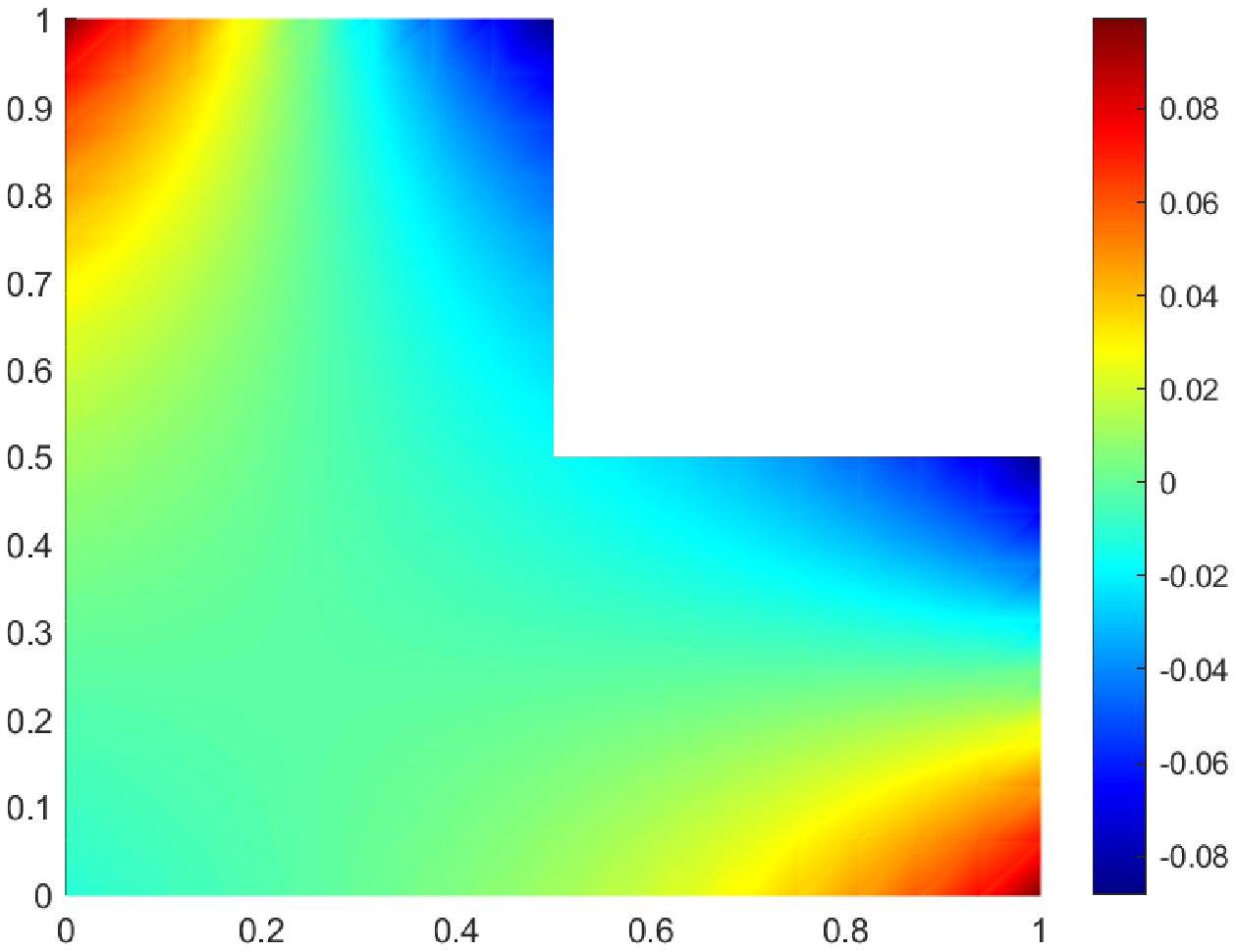}
\caption{$\Omega=(0,1)\times (0,1)\setminus[\frac{1}{2},1]\times [\frac{1}{2},1]$, $k=1$, $\gamma(h)=h^{0.1}$, $h=1/16$, the fourth eigenfunction.}
\label{figure5.8}
\end{minipage}
\end{figure}

We show the numerical results for different combinations of $k$ and $\gamma(h)$ in Tables \ref{table5.5} - \ref{table5.7}, it becomes evident that the eigenvalues obtained using the WG method serve as asymptotic lower bound approximations for the exact eigenvalues of the Steklov eigenvalue problem $(\ref{2.1})$ on the L-shaped domain $\Omega=(0,1)\times (0,1)\setminus[\frac{1}{2},1]\times [\frac{1}{2},1]$. This observation is consistent with the theoretical analysis provided in Theorem \ref{theorem2.15}. This shows the accuracy of our theoretical analysis.

\bibliographystyle{siam}
\bibliography{library}

\end{document}